\documentclass[a4paper,reqno,oneside,12pt]{amsart}
\usepackage{eucal,dsfont,url,cite}

\allowdisplaybreaks[4]

\newcommand{\dd}{\,\mathrm{d}}

\renewcommand{\Hat}{\widehat}

\newcommand{\RR}{\mathbb{R}}

\newcommand{\NN}{\mathbb{N}}

\newcommand{\zero}{\mathds{O}}
\newcommand{\one}{\mathds{1}}

\newcommand{\cO}{\mathcal{O}}
\newcommand{\cD}{\mathcal{D}}
\newcommand{\cG}{\mathcal{G}}
\newcommand{\cH}{\mathcal{H}}
\newcommand{\cL}{\mathcal{L}}

\newcommand{\cQ}{\mathcal{Q}}

\newcommand{\supp}{\mathrm{supp}}

\newtheorem{theorem}{Theorem}
\newtheorem{prop}[theorem]{Proposition}

\newtheorem{lemma}[theorem]{Lemma}

\theoremstyle{definition}

\theoremstyle{remark}
\newtheorem{rem}[theorem]{\bf Remark}

\newenvironment{nouppercase}{\renewcommand{\uppercasenonmath}[1]{}}{}

\begin{document}

\title[]{\large\bf Strong coupling asymptotics for $\delta$-interactions supported by curves with cusps}

\author{\large Brice Flamencourt}

%
\author{\large Konstantin Pankrashkin}
\address{\normalfont  Laboratoire de Math\'ematiques d'Orsay, Univ. Paris-Sud, CNRS, Universit\'e Paris-Saclay, 91405 Orsay, France \raggedright}
\email{brice.flamencourt@math.u-psud.fr, konstantin.pankrashkin@math.u-psud.fr \raggedright}

\subjclass[2010]{35J50, 35P20, 35Q40, 47A75}
\keywords{Schr\"odinger operator; $\delta$-potential; boundary with a cusp; eigenvalue asymptotics; dimension reduction; effective operator}

\dedicatory{Dedicated to the memory of Johannes F. Brasche \textup{(}1956--2018\textup{)}}

\begin{nouppercase}
\maketitle
\end{nouppercase}

\begin{abstract}
Let $\Gamma\subset \mathbb{R}^2$ be a simple closed curve
which is smooth except at the origin, at which it has a power cusp and
coincides with the curve $|x_2|=x_1^p$ for some $p>1$. We study the eigenvalues
  of the Schr\"odinger operator $H_\alpha$ with the attractive
$\delta$-potential of strength $\alpha>0$ supported by $\Gamma$, which is defined by its quadratic form
\[
H^1(\mathbb{R}^2)\ni u\mapsto \iint_{\mathbb{R}^2} |\nabla u|^2\,\mathrm{d}x-\alpha\int_\Gamma u^2\, \mathrm{d}s,
\]
where $\mathrm{d}s$ stands for the one-dimensional Hausdorff measure on $\Gamma$. It is shown that if $n\in\NN$ is fixed
and $\alpha$ is large, then the well-defined $n$th eigenvalue $E_n(H_\alpha)$ of $H_\alpha$ behaves as
\[
E_n(H_\alpha)=-\alpha^2 + 2^{\frac{2}{p+2}} \mathcal{E}_n \,\alpha^{\frac{6}{p+2}} + \mathcal{O}(\alpha^{\frac{6}{p+2}-\eta}),
\]
where the constants $\mathcal{E}_n>0$ are the eigenvalues of an explicitly given one-dimensional Schr\"odinger operator determined by the cusp, and $\eta>0$.
Both  main and secondary terms in this asymptotic expansion are different from what was observed previously for the cases when~$\Gamma$ is smooth or piecewise smooth with non-zero angles.
 \end{abstract}



\section{\bf Introduction} 

Schr\"odinger operators with singular interactions supported by submanifolds
represent an important class of models in mathematical physics,
and they have been the subject of an intensive study during the last decades.
In the present work we deal with two-dimensional operators, so we
assume that $\Gamma$ is a metric graph embedded in the Euclidean space $\RR^2$, and we will
be interested in the spectral study of the operators formally written as 
\[
H_\alpha:=-\Delta - \alpha\delta(x-\Gamma)
\]
with $\delta$ being the Dirac distribution and $\alpha>0$ being the coupling constant. Such operators describe the
motion of particles confined to the graph $\Gamma$ but allowing for a quantum tunneling between
its different parts. The above definition is made rigorous
by considering first the quadratic form
\[
H^1(\RR^2)\ni u\mapsto h_\alpha(u,u):=\iint_{\RR^2} |\nabla u|^2\dd x-\alpha\int_\Gamma u^2\dd s,
\]
where $\dd s$ is the one-dimensional Hausdorff measure on $\Gamma$. Under suitable regularity
assumptions on $\Gamma$ (e.g. a finite union of bounded Lipschitz curves) the quadratic form $h_\alpha$ is closed and semibounded from below,
and, hence, generate in a canonical way a unique self-adjoint operator $H_\alpha$ in $L^2(\RR^2)$
whose domain is contained in $H^1(\RR^2)$ and such that
\[
\iint_{\RR^2} u\, H_\alpha u\,\dd x=h_\alpha(u,u)
\]
for any function $u$ in the domain. In the informal language, the  operator $H_\alpha$
is the distributional Laplacian in $\RR^2\setminus\Gamma$ with interface conditions $[\partial u]+\alpha u=0$ on $\Gamma$, where
$[\partial u]$ denotes a suitably defined jump of the normal derivative of $u$ on $\Gamma$,
see e.g. \cite{BLL,BEKS} for a more detailed discussion.

The well-known review paper \cite{leaky} provides an introduction to the topic and proposes a number of research directions. An interesting problem setting is provided by the strong coupling regime, i.e. the case $\alpha\to+\infty$. It can be easily seen that
the lowest eigenfunctions of $H_\alpha$ concentrate exponentially near $\Gamma$, so that one might expect that
an ``effective operator'' on $\Gamma$ governing the spectral behavior could come in play.
This was first proved in \cite{EY} for the case when $\Gamma$ is a $C^4$-smooth loop:
for any fixed $n\in\NN$ the operator $H_\alpha$ admits at least $n$ negative eigenvalues
if $\alpha$ is sufficiently large, and the $n$th eigenvalue $E_n(H_\alpha)$ behaves as
\begin{equation}
  \label{EY}
E_n(H_\alpha)=-\tfrac{1}{4}\,\alpha^2+E_n(P)+\cO\big( \tfrac{\log\alpha}{\alpha}\big),
\end{equation}
where $P$ is the operator on $L^2(\Gamma)$ acting in the arc-length parametrization as $f\mapsto-f''-\frac{1}{4}\, \gamma^2 f$
with $\gamma$ being the curvature.
A similar result holds for finite open arcs as well~\cite{EP}.
To our knowledge, no sufficiently detailed analysis for non-smooth $\Gamma$ was carried out
so far. Being based on the general machinery for
problems with corners \cite{BND,BP,kh2} one might expect that if $\Gamma$ is piecewise smooth with non-zero angles, then
at least several lowest eigenvalues behave as $E_n(H_\alpha)\simeq -\mu_n \alpha^2$ as $\alpha\to+\infty$, where
$\mu_n\in (\frac{1}{4},1)$ are spectral quantities associated with
some model operators (so-called star leaky graphs) whose exact values are
not known: we refer to \cite{bew2,DR,EL,lot,P15} for a number of estimates.

It seems that no work analyzed the case of non-Lipschitz $\Gamma$, and we make the first step
in this direction in the present text by considering curves with power cusps.
More precisely, we assume that $\Gamma$ is a Jordan curve satisfying $0\in\Gamma$ and the following two conditions:
\begin{align}
&\text{$\Gamma$ is $C^4$-smooth at all points except at the origin,}\nonumber\\
&\text{there exist $\varepsilon_0>0$ and $p>1$ such that} \nonumber\\
&\qquad \Gamma \cap (-\varepsilon_0,\varepsilon_0)^2=\big\{
(x_1,x_2): \, x_1 \in(0,\varepsilon_0), \, |x_2|=x_1^p \big\}.
   \label{gam2}
\end{align}
The value $p$ is indeed unique.
It is easily seen that the essential spectrum of $H_\alpha$ covers the half-axis $[0,+\infty)$
and that for any $\alpha>0$ the discrete spectrum is non-empty and finite.
Our result on the asymptotics of individual eigenvalues of $H_\alpha$ for large $\alpha$
involves an auxiliary one-dimensional operator $A$ in $L^2(0,+\infty)$ acting as
\[
(Af)(x)=-f''(x)+x^p f(x)
\]
on the functions $f$ satisfying the Dirichlet condition $f(0)=0$. It is directly
seen that $A$ has compact resolvent and that all its eigenvalues $E_n(A)$ are strictly positive and simple.

\begin{theorem} \label{thm-main1}
For any fixed $n\in\NN$ one has, as $\alpha$ tends to $+\infty$,
\[
E_n(H_\alpha) =-\alpha^2+2^{\frac{2}{p+2}} E_n(A) \,\alpha^{\frac{6}{p+2}} + \cO(\alpha^{\frac{6}{p+2}-\eta})
\]
where $\eta:=\min\big\{ \frac{p-1}{2(p+2)}, \frac{2(p-1)}{(p+1)(p+2)}\big\}>0$.
\end{theorem}

\begin{rem}
For the quadratic cusp, $p=2$, the eigenvalues $E_n(A)$ can be computed explicitly. The operator $A$ in this case is unitary equivalent to the restriction
of the harmonic oscillator to the odd functions, and its eigenvalues are the usual harmonic oscillator eigenvalues with even numbers, i.e.
$E_n(A)=4n-1$ for any $n\in\NN$. Hence, the asymptotics of Theorem~\ref{thm-main1} takes the very explicit form
\[
E_n(H_\alpha) =-\alpha^2+(4n-1)\sqrt{2} \,\alpha^{\frac{3}{2}} + \cO(\alpha^\frac{11}{8}).
\]
We are not aware of other values of $p>1$ admitting a simple expression for the eigenvalues of $A$.
\end{rem}

\begin{rem}
Both main and secondary terms in the result of Theorem~\ref{thm-main1} are different  from
the asymptotics \eqref{EY} for the smooth curves
and from the expectations for the curves with non-zero angles.
In particular, the distance between the individual eigenvalues
is of order $\alpha^k$,
where the power $k=\frac{6}{p+2}$ can be given any value between $0$ and $2$ by a suitable choice of $p\in(1,+\infty)$. Such a control
of the eigenvalue gap asymptotics represents a new feature of the model, which is not observed for $\delta$-potentials
supported by curves of a higher regularity.
Nevertheless we recall that similar effects can be seen in other boundary eigenvalue problems by a suitable control 
of the boundary curvature, see e.g. \cite{FS,PP}.
\end{rem}

\begin{rem}
One should remark that the presence of a singularity does not involve any problem with the semiboundedness of the form $h_\alpha$,
and arbitrary values of $p$ are allowed due to the fact that both sides of $\Gamma$ are involved. In fact,
this directly follows from the fact that $\Gamma$ can be decomposed into two smooth open arcs,
and the $L^2$-trace of a function from $H^1(\RR^2)$ to such an arc is well-defined.
This is in contrast with the one-sided Robin problems for the Laplacian in a domain surrounded by $\Gamma$,
for which the cusp is not allowed to be very sharp: see e.g. \cite{KP}
for the study of the eigenvalues and \cite{mp,nt} for the issues concerning the definition of the operator.
\end{rem}

The proof of Theorem~\ref{thm-main1} is almost entirely based on the min-max tools for the study of the eigenvalues: we recall them in Section~\ref{sec-prel}.
We first apply some truncations in order to localize the problem near the cusp and then extend it to a suitable half-place and
rescale it in order to have a semiclassical formulation admitting a more explicit analysis (Section~\ref{sec-half}).
The resulting problem in the half-plane is analyzed by considering first the action of the operator in one of the variables
and then by showing that only the projection onto the lowest mode contribute to the individual eigenvalues.
At some points the problem shows a number of similarities to the case when $\Gamma$ is a sharply broken line \cite{DR},
and we were able to use a part of that analysis. The overall proof scheme is rather classical, see e.g \cite{FS}, but
a big number of various new technical ingredients and adapted variables are required  in order to carry out the complete study.
In Section~\ref{secup} we show
the upper bound for $E_n(H_\alpha)$, which is rather straightforward. The lower bound is obtained in Section~\ref{seclow},
and is much more demanding, both for the dimension reduction and for the analysis of the resulting one-dimensional effective operator.

\medskip

The present work is dedicated to the memory of Johannes F. Brasche (1956--2018). His first works on Schr\"odinger operators with measure potentials \cite{ABR,BEKS} served as a basis for the rigorous mathematical study of a large class of quantum-mechanical models,
and the works of last years on large coupling convergence \cite{BBB} suggested a far-reaching abstract generalization of strongly coupled $\delta$-interactions, which will certainly lead to further progress in the domain.

\section{\bf Preliminaries}\label{sec-prel}

We will recall some notation and
basic facts on the min-max principle for the eigenvalues of self-adjoint operators.

In this paper we only deal with real-valued operators, so we prefer to work with real Hilbert spaces. Let $\cH$  be a Hilbert space and $u\in\cH$,
then we denote by $\|u\|_{\cH}$ the norm of $u$. For a linear operator $T$ we denote $\cD(T)$ its domain.
If the operator $T$ is self-adjoint and semibounded from below, then $\cQ(T)$ denotes the domain of its bilinear form, and the value of the bilinear form on $u,v\in \cQ(T)$ will be denoted by $T[u,v]$. For $n\in\NN:=\{1,2,3,\dots\}$, by $E_n(T)$ we denote the $n$th discrete eigenvalue of $T$ (if it exists) when enumerated in the non-decreasing order and taking the multiplicities into account.

Let  $\mathcal H$ be an infinite-dimensional Hilbert space and $T$ be a lower semibounded self-adjoint operator in $\cH$.
If $T$ is with compact resolvent, we set $\Sigma:=+\infty$, otherwise let $\Sigma$ denote the bottom of the essential spectrum of $T$.
The $n$th \emph{Rayleigh quotient}  $\Lambda_n(T)$ of $T$ is defined by
\[
\Lambda_n(T):= \inf_{\substack{\cL\subset \cQ(T)\\ \dim \cL=n}} \,\,\sup_{u \in \cL\setminus\{0\}} \frac{T[u,u] }{\|u\|^2_\cH}.
\]
The well-known min-max principle, see e.g. Theorem~5 in Section~10.2 of \cite{bs}, states that
one and only one of the following two assertions is true:
\begin{itemize}
\item[(a)] $\Lambda_n(T)<\Sigma$ and $E_n(T)=\Lambda_n(T)$.
\item[(b)] $\Lambda_n(T)=\Sigma$ and $\Lambda_m(T)=\Lambda_n(T)$ for all $m\ge n$.
\end{itemize}
In what follows we will actively work with the Rayleigh quotients of various operators instead of eigenvalues
as the former are easier to deal with.
The passage from the Rayleigh quotients to the eigenvalues will be done at suitable points by simply
checking that the values are below the essential spectrum.

One of the most classical applications of the min-max principle is recalled in the next assertion (the proof is by a direct application of the definition). It will be used systemically through the whole text.

\begin{prop}\label{propineq}
Let $T$ and $T'$ be lower semibounded self-adjoint operators in infinite-dimensional Hilbert spaces $\cH$ and $\cH'$ respectively.
Assume that there exists a linear map $J:\cQ(T)\to\cQ(T')$ such that
\[
\|Ju\|_{\cH'}=\|u\|_{\cH}, \quad T'[Ju,Ju]\le T[u,u] \text{ for all } u\in\cQ(T).
\]
Then for any $n\in\NN$ there holds $\Lambda_n(T')\le \Lambda_n(T)$.
\end{prop}

At the last steps of the proof of Theorem \ref{thm-main1} we will also
need the following result, which is a slight reformulation of \cite[Lemma 2.1]{Ex03} or of~\cite[Lemma~2.2]{post}. As some details are different, we prefer to give a complete proof, which is quite short.

\begin{prop} \label{Comp}
Let $\cH$, $\cH'$ be two infinite-dimensional Hilbert spaces and
$T$ be a \emph{non-negative} self-adjont operator in $\cH$ and $T'$ be a lower semibounded
self-adjoint operator in $\cH'$.
Assume that there exist a linear map $J: \cQ(T)\to \cQ(T')$
and non-negative numbers $\delta_1$ and $\delta_2$ such that
for all $u\in\cQ(T)$ there holds
\begin{align*}
\|u \|_\cH^2  - \|Ju\|_{\cH'}^2&\le\delta_1 \big(T[u,u] +  \|u \|_\cH^2\big), \\
T'[Ju,Ju] - T[u,u] &\le \delta_2 \big(T[u,u] +  \|u \|_\cH^2\big),
\end{align*}
and that for some $n\in\NN$ one has the strict inequality
\begin{equation}
   \label{asmp1}
\delta_1 \big(\Lambda_n(T)+1\big)<1,
\end{equation}
then
\[
\Lambda_n(T') \leq \Lambda_n(T) + 
\frac{\big(\delta_1\Lambda_n(T)+\delta_2\big)\big(\Lambda_n(T)+ 1\big)}{1-\delta_1\big(\Lambda_n(T)+ 1\big)}.
\]
\end{prop}

\begin{proof}
During the proof we abbreviate $\lambda_n:=\Lambda_n(T)$.
By \eqref{asmp1}, for any sufficiently small $\varepsilon>0$ one has
\begin{equation}
   \label{eps1}
\delta_1 (\lambda_n + 1+\varepsilon)<1.
\end{equation}
In view of the definition of $\lambda_n$, one can find an $n$-dimensional subspace $F\subset \cQ(T)$
such that $T[u,u] \le (\lambda_n + \varepsilon)\|u \|_\cH^2$ for all $u\in F$.
Therefore, for  any $u \in F$ one has 
\[
\| J u \|^2_{\cH'} \ge (1-\delta_1)\| u \|_\cH^2 - \delta_1 T [u,u]
\ge \big(1-\delta_1(\lambda_n + 1+\varepsilon)\big)\|u \|_\cH^2.
\]
The first factor on the right-hand side is strictly positive by \eqref{eps1}, and it follows that
$J:F\to J(F)$ is injective. In particular, $\dim J(F)=n$.
Therefore, for $u\in F\setminus\{0\}$ one has $Ju\ne 0$ and
\begin{align*}
\frac{T'[Ju,Ju]}{\| Ju\|_{\cH'}^2} &
\le \frac{T[u,u]+ \delta_2\big(T[u,u]+ \|u\|_\cH^2\big)}{\|Ju\|_{\cH'}^2} \\
&\le \frac{T[u,u] + \delta_2\big(T[u,u]+ \|u\|_\cH^2\big)  }{  \big(1-\delta_1(\lambda_n + 1+\varepsilon)\big)\|u \|_\cH^2  }
 \le \frac{ \lambda_n+\varepsilon + \delta_2(\lambda_n+ 1+\varepsilon)}{1-\delta_1(\lambda_n + 1+\varepsilon)}\\
&=\lambda_n + \frac{ \lambda_n+\varepsilon + \delta_2(\lambda_n+ 1+\varepsilon) -\lambda_n\big(1-\delta_1(\lambda_n + 1+\varepsilon)\big)}{1-\delta_1(\lambda_n + 1+\varepsilon)}\\
&=\lambda_n + \frac{ \varepsilon + (\delta_1\lambda_n+\delta_2)(\lambda_n+ 1+\varepsilon)}{1-\delta_1(\lambda_n + 1+\varepsilon)}.
\end{align*}
Due to the definition of $\Lambda_n(T')$ one has
\begin{align*}
\Lambda_n(T') &\leq \sup_{v \in J(F)\setminus\{0\}} \frac{T'[v,v]}{\|v\|_{\cH'}^2}
=\sup_{u \in F\setminus\{0\}} \frac{T'[Ju,Ju]}{\|Ju\|_{\cH'}^2}\\
&\le \lambda_n + \frac{ \varepsilon + (\delta_1\lambda_n+\delta_2)(\lambda_n+ 1+\varepsilon)}{1-\delta_1(\lambda_n + 1+\varepsilon)},
\end{align*}
and the claim follows by sending $\varepsilon$ to zero.
\end{proof}

\section{\bf Reduction to a problem in a moving half-plane}\label{sec-half}

We first apply some truncations in order to obtain a model problem which only takes into account the cusp and neglects the rest of~$\Gamma$.
For $\varepsilon>0$ we denote 
\[
\Gamma_{\varepsilon}:=\big\{(x_1,x_2): \, x_1\in(0,\varepsilon), |x_2|=x_1^p\big\}
\]
and consider the half-plane
\[
\Omega_\varepsilon:=(-\infty,\varepsilon)\times \RR.
\]
One clearly has $\Gamma_\varepsilon\subset\Omega_\varepsilon$, and by $H_{\alpha,\varepsilon}$ we denote the self-adjoint operator in $L^2(\Omega_\varepsilon)$ given by
\begin{equation*}
H_{\alpha,\varepsilon}=\iint_{\Omega_\varepsilon} |\nabla u|^2\dd x -\alpha \int_{\Gamma_\varepsilon} u^2\dd s,
\quad
\cQ(H_{\alpha,\varepsilon})=H^1_0(\Omega_\varepsilon).
\end{equation*}
We start with the following result, taking $\varepsilon_0$ from (\ref{gam2}):
\begin{lemma}\label{lem4}
Let $\varepsilon\in(0,\varepsilon_0)$ and $n\in\NN$. Assume that 
\begin{equation}
\label{asscc}
\text{for some $c>\tfrac{1}{4}$ there holds $\Lambda_n(H_{\alpha,\varepsilon})\le -c\alpha^2$ for large $\alpha>0$,}
\end{equation}
then $\Lambda_n(H_\alpha)=\Lambda_n(H_{\alpha,\varepsilon})+\cO(1)$ for $\alpha\to +\infty$.
\end{lemma}

\begin{proof}
The proof will be in two steps. We first reduce the problem to a bounded neighborhood of the origin, and then to the half-plane $\Omega_\varepsilon$, as the latter is easier to analyze.

For $\varepsilon>0$ denote $\square_\varepsilon:=(-\varepsilon,\varepsilon)^2$, then the assumption \eqref{gam2} rewrites as
\[
\text{there exists $\varepsilon_0>0$ such that $\Gamma\cap\square_{\varepsilon_0}=\Gamma_{\varepsilon_0}$,}
\]
and then for any $\varepsilon\in(0,\varepsilon_0)$ one has $\Gamma\cap\square_{\varepsilon}=\Gamma_{\varepsilon}$ as well; we remark that the condition \eqref{gam2} implies
$\varepsilon_0\le1$.

{}From now on let us pick some $\varepsilon\in (0,\varepsilon_0)$ and
let $\chi_1,\chi_2\in C^\infty(\RR^2)$ such that $\chi_1^2+\chi_2^2=1$ and
\[
\chi_1=1 \text{ in } \square_{\frac{\varepsilon}{2}}, \quad \chi_1=0 \text{ outside } \square_{\varepsilon}.
\]
An easy computation shows that for any $u\in \cQ(H_\alpha)\equiv H^1(\RR^2)$ one has
\begin{align}
H_\alpha[u,u]&=H_\alpha[\chi_1u,\chi_1 u]+H_\alpha[\chi_2u,\chi_2 u]-\int_{\RR^2} \big(|\nabla \chi_1|^2+|\nabla \chi_2|^2\big)\, u^2\dd x\nonumber\\
&\ge H_\alpha[\chi_1u,\chi_1 u]+H_\alpha[\chi_2u,\chi_2 u]-C\|u\|^2_{L^2(\RR^2)}, \label{eq22}
\end{align}
where $C=\big\||\nabla \chi_1|^2+|\nabla \chi_2|^2\big\|_\infty$.

Denote by $D_{\alpha,\varepsilon}$ the self-adjoint operator in $L^2(\square_\varepsilon)$ given by
\[
D_{\alpha,\varepsilon}[u,u]=\iint_{\square_\varepsilon} |\nabla u|^2\dd x -\alpha \int_{\Gamma_\varepsilon} u^2\dd s,
\quad
\cQ(D_{\alpha,\varepsilon})=H^1_0(\square_\varepsilon).
\]
Due to $\supp \chi_1\subset \square_\varepsilon$ we have
\[
\chi_1u\in \cQ(D_{\alpha,\varepsilon}), \quad H_\alpha[\chi_1u,\chi_1 u]=D_{\alpha,\varepsilon}[\chi_1u,\chi_1 u].
\]
On the other hand, by the initial assumption of $\Gamma$ ($C^4$-smoothness except at the origin)
one can find a $C^4$-smooth Jordan curve $\Gamma'$ which coincides
with $\Gamma$ outside $\square_{\frac{\varepsilon}{2}}$. Denote by $H'_\alpha$ the self-adjoint operator in $L^2(\RR^2)$
given by
\[
H'_\alpha[u,u]=\iint_{\RR^2} |\nabla u|^2\,\dd x-\alpha\int_{\Gamma'} u^2\, \dd s, \quad \cQ(H'_\alpha)=H^1(\RR^2).
\]
As $\supp \chi_2\cap \square_{\frac{\varepsilon}{2}}=\emptyset$, one has $H_\alpha[\chi_2u,\chi_2 u]=H'_\alpha[\chi_2u,\chi_2 u]$,
and the inequality \eqref{eq22} takes the form
\begin{equation}
    \label{eq23}
H_\alpha[u,u]+C\|u\|^2_{L^2(\RR^2)}\ge D_{\alpha,\varepsilon}[\chi_1u,\chi_1 u]+H'_\alpha[\chi_2u,\chi_2 u].
\end{equation}
Noting that $J: L^2(\RR^2)\ni u\mapsto (\chi_1 u,\chi_2 u)\in L^2(\square_\varepsilon)\oplus L^2(\RR^2)$
is isometric and that \eqref{eq23} can be rewritten as
\[
H_\alpha[u,u]+C\|u\|^2_{L^2(\RR^2)}\ge (D_{\alpha,\varepsilon}\oplus H'_\alpha)[Ju,Ju],
\]
we conclude by the min-max principle (Proposition~\ref{propineq}) that
\[
\Lambda_n(H_\alpha)\ge \Lambda_n(D_{\alpha,\varepsilon}\oplus H'_\alpha)-C \text{ for all } n\in\NN, \, \alpha>0.
\]
As discussed in the introduction, see e.g. Eq.~\eqref{EY}, due to the smoothness of $\Gamma'$, for some $C_0>0$ one has $H'_\alpha\ge -\frac{1}{4}\,\alpha^2-C_0$ for large $\alpha>0$. Hence, if
\begin{equation}
   \label{eqcd}
\text{for some $c>\tfrac{1}{4}$ there holds $\Lambda_n(D_{\alpha,\varepsilon})\le -c\alpha^2$ for large $\alpha>0$,}
\end{equation}
then $\Lambda_n(D_{\alpha,\varepsilon}\oplus H'_\alpha)=\Lambda_n(D_{\alpha,\varepsilon})$, and then
$\Lambda_n(H_\alpha)\ge \Lambda_n(D_{\alpha,\varepsilon})-C$ for large $\alpha>0$. On the other hand, by the min-max principle
one directly has $\Lambda_n(H_\alpha)\le \Lambda_n(D_{\alpha,\varepsilon})$. Therefore, the assumption \eqref{eqcd} implies
\begin{equation}
   \label{eqhd}
\Lambda_n(H_\alpha)= \Lambda_n(D_{\alpha,\varepsilon})+\cO(1) \text{ for } \alpha\to+\infty.
\end{equation}

Now we need to pass from $D_{\alpha,\varepsilon}$ to $H_{\alpha,\varepsilon}$, which is done in a very similar way.
First, by the min-max principle we have
\begin{equation}
   \label{eqhhdd}
\Lambda_n(H_{\alpha,\varepsilon})\le \Lambda_n(D_{\alpha,\varepsilon})
\end{equation}
for any $\alpha>0$. Furthermore, let us pick  $\xi_1,\xi_2\in C^\infty(\RR^2)$
such that $\xi_1^2+\xi_2^2=1$ and
\begin{align*}
\xi_1(x)&=1 \text{ for } x\in(0,+\infty)\times (-\varepsilon^p,\varepsilon^p),\\
\xi_1(x)&=0 \text{ for } x\notin(-\varepsilon,+\infty)\times (-\varepsilon,\varepsilon).
\end{align*}
For any $u\in \cQ(H_{\alpha,\varepsilon})$ we have then, with $W(x):=|\nabla \xi_1|^2+|\nabla \xi_2|^2\le C'$,
\begin{align*}
H_{\alpha,\varepsilon}[u,u]&=H_{\alpha,\varepsilon}[\xi_1u,\xi_1 u]+H_{\alpha,\varepsilon}[\xi_2u,\xi_2 u]-\iint_{\Omega_\varepsilon} W\, u^2\dd x\nonumber\\
&\equiv D_{\alpha,\varepsilon}[\xi_1u,\xi_1 u]+ \iint_{\Omega_\varepsilon} \big|\nabla(\xi_2 u )\big|^2 \dd x-\iint_{\Omega_\varepsilon} W\, u^2\dd x\nonumber\\
&\ge D_{\alpha,\varepsilon}[\xi_1u,\xi_1 u]- C'\|u\|^2_{L^2(\Omega_\varepsilon)}.
\end{align*}
As in the first part of the proof, this implies
\begin{equation}
   \label{eq24}
\Lambda_n(H_{\alpha,\varepsilon})\ge \Lambda_n(D_{\alpha,\varepsilon}\oplus \zero)-C'
\end{equation}
with $\zero$ being the zero operator in $L^2(\Omega_\varepsilon)$. Let \eqref{asscc} hold, then by \eqref{eq24} we also have
$\Lambda_n(D_{\alpha,\varepsilon}\,\oplus\, \zero)\le -c\alpha^2$ for large $\alpha$. Then
$\Lambda_n(D_{\alpha,\varepsilon}\,\oplus\, \zero)=\Lambda_n(D_{\alpha,\varepsilon})$, and \eqref{eqcd} holds, which implies the estimate~\eqref{eqhd}.
At the same time, Eq. \eqref{eq24} reads now as $\Lambda_n(H_{\alpha,\varepsilon})\ge \Lambda_n(D_{\alpha,\varepsilon})-C'$,
and together with \eqref{eqhhdd} we arrive at $\Lambda_n(D_{\alpha,\varepsilon})=\Lambda_n(H_{\alpha,\varepsilon})+\cO(1)$
for large $\alpha$. Substituting this estimate into~\eqref{eqhd} we prove the claim.
\end{proof}

Let us apply an additional scaling in order to pass to the semiclassical framework. 
For $h>0$ and $b>0$ consider the self-adjoint operator $F_{h,b}$ in $L^2(\Omega_b)$ defined for
$\cQ(F_{h,b})=H^1_0(\Omega_b)$ by
\begin{multline*}
F_{h,b}[u,u]=\iint_{\Omega_b} \big(h^2(\partial_1 u)^2+ (\partial_2 u)^2\big)\dd x\\
-\int_0^b \sqrt{1+p^2h^2s^{2(p-1)}} \big(u(s,s^p)^2+u(s,-s^p)^2\big)\dd s.
\end{multline*}
\begin{lemma}\label{lemx}
For any $\varepsilon>0$ and $\alpha>0$ and $n\in\NN$ one has
\begin{equation*}
\Lambda_n(H_{\alpha,\varepsilon})=\alpha^2 \Lambda_n(F_{h,b}) \text{ for } h=\alpha^\frac{1-p}{p}, \quad b=\varepsilon \alpha^{\frac{1}{p}}\equiv \varepsilon h^{\frac{1}{1-p}}.
\end{equation*}
\end{lemma}

\begin{proof}
We prefer to give a detailed explicit computation. Consider the unitary operator $\Theta:L^2(\Omega_b)\to L^2(\Omega_\varepsilon)$
given by
\[
(\Theta u)(x_1,x_2)=\alpha^{\frac{1}{2}(\frac{1}{p}+1)}\,u(\alpha^{\frac{1}{p}} x_1,\alpha x_2),
\]
then $\Theta \cQ(F_{h,b})=\cQ(H_{\alpha,\varepsilon})$.
By writing the one-dimensional Hausdorff measure on $\Gamma_\varepsilon$
in an explicit form, for any $u\in \cQ(H_{\alpha,\varepsilon})$ we have
\begin{multline*}
H_{\alpha,\varepsilon}[u,u]
=\iint_{\Omega_\varepsilon} \big[ (\partial_1 u)^2+ (\partial_2 u)^2\big]\dd x\\
-\alpha\int_0^\varepsilon \sqrt{1+p^2s^{2p-2}}\big(u(s,s^p)^2+u(s,-s^p)^2\big)\dd s.
\end{multline*}
Then for any $v\in \cQ(F_{h,b})$ one obtains
\begin{align*}
H_{\alpha,\varepsilon}[\Theta v,\Theta v]
=&\,\alpha^{\frac{1}{p}+1}\iint_{\Omega_\varepsilon} \Big[ \alpha^{\frac{2}{p}}\partial_1 v (\alpha^{\frac{1}{p}} x_1,\alpha x_2)^2\\
&\qquad \qquad + \alpha^2\partial_2 v (\alpha^{\frac{1}{p}} x_1,\alpha x_2)^2\Big]\dd x_1\dd x_2\\
&-\alpha^{\frac{1}{p}+2}\int_0^\varepsilon \sqrt{1+p^2s^{2(p-1)}}\big(v(\alpha^{\frac{1}{p}}s,\alpha s^p)^2\\
&\qquad+v(\alpha^{\frac{1}{p}}s,-\alpha s^p)^2\big)\dd s.
\end{align*}
Using the new variables $y_1=\alpha^{\frac{1}{p}}x_1$, $x_2=\alpha y_2$, $t=\alpha^{\frac{1}{p}} s$ we rewrite it as
\begin{align*}
H_{\alpha,\varepsilon}[\Theta v,\Theta v]
&=\iint_{\Omega_{\varepsilon \alpha^{\frac{1}{p}}}} \Big[ \alpha^{\frac{2}{p}}\partial_1 v (y_1,y_2)^2+ \alpha^2\partial_2 v (y_1,y_2)^2\Big]\dd y_1\dd y_2\\
&\quad -\alpha^2\int_0^{\varepsilon\alpha^{\frac{1}{p}}} \sqrt{1+p^2\alpha^{\frac{2-2p}{p}} s^{2p-2}}\big(v(t,t^p)+v(t,-t^p)\big)\dd t\\
&=\alpha^2 F_{h,b}[v,v],
\end{align*}
which shows that $H_{\alpha,\varepsilon}$ is unitarily equivalent to $\alpha^2 F_{h,b}$.
\end{proof}
By combining Lemma~\ref{lem4} with Lemma~\ref{lemx} we arrive at the following reformulation:
\begin{lemma}\label{lemyy}
Let $\varepsilon>0$, $h_0>0$, $n\in\NN$ be such that
\begin{equation}
  \label{asc1}
\Lambda_n(F_{h,\varepsilon h^{\frac{1}{1-p}}})\le -c \text{ for all $h\in(0,h_0)$ and some $c>\tfrac{1}{4}$.}
\end{equation}
Then $\Lambda_n(H_\alpha)=\alpha^2\Lambda_n(F_{h,\varepsilon h^{\frac{1}{1-p}}})+\cO(1)$ for 
$h:=\alpha^\frac{1-p}{p}$ and $\alpha\to+\infty$.
\end{lemma}

\section{\bf Upper bound}\label{secup}

\subsection{Reduction to a one-dimensional effective operator}\label{secgh}

For some $k>0$, to be chosen later, denote
\[
\Omega'_h:=(0,h^k) \times \RR
\]
and denote by $G_h$ the self-adjoint operator in $L^2(\Omega'_h)$ given by
\[
G_h[u,u]=\iint_{\Omega'_h} \big(h^2(\partial_1 u)^2+ (\partial_2 u)^2\big)\dd x-\int_0^{h^k}\big(u(s,s^p)^2+u(s,-s^p)^2\big)\dd s
\]
and $\cQ(G_h)=H^1_0(\Omega'_h)$. For sufficiently small $h>0$ one has the inclusion $\Omega'_h\subset\Omega_{\varepsilon h^{\frac{1}{1-p}}}$, and for $u\in H^1_0(\Omega'_h)$ we denote $u_0$ its extension by zero to $\Omega_{\varepsilon h^{\frac{1}{1-p}}}$,
then $F_{h,b}[u_0,u_0]\le G_h[u,u]$. It follows by the min-max principle that:
\begin{lemma}\label{lemx2}
For any $\varepsilon>0$ there exists $h_0>0$ such that for  $h\in(0,h_0)$ and $n\in\NN$ there holds $\Lambda_n(F_{h,\varepsilon h^{\frac{1}{1-p}}})\le \Lambda_n(G_h)$.
\end{lemma}

In order to study $G_h$ we will use some facts on a simple one-dimensional operator $T_{x}$, $x>0$,
which is the self-adjoint operator in $L^2(\RR)$ given by
\begin{equation}
   \label{eqtx}
T_x[f,f] = \int_{\RR} f'(y)^2\dd y -\big(f(x)^2+ f(-x)^2\big), \quad \cQ(T_x)=H^1(\RR).
\end{equation}
We recall some simple properties of $T_x$ established in \cite[Proposition 2.3]{DR}. The bottom of the spectrum of $T_x$ is a simple isolated
eigenvalue, which we denote by $\sigma(x)$ due to its special role in what follows,
\[
\sigma(x):=\Lambda_1(T_x), \quad x>0,
\]
and we denote by $\Psi_x$ the respective eigenfunction chosen $L^2$-normalized and positive. We will use the following properties of their dependence on $x>0$:
\begin{prop}\label{prop1d} The following holds:
\begin{itemize}
\item[(a)] $-1 < \sigma(x) < -\frac{1}{4}$ for all $x\in(0,+\infty)$,
\item[(b)] $\sigma$ is non-decreasing,
\item[(c)] $\sigma(x) = -1 + 2x + \cO(x^2)$ for $x\to0^+$,
\item[(d)] the function $x\mapsto\|\partial_x \Psi_x \|_{L^2(\RR)}$ is bounded on $(0,+\infty)$,
\item[(e)] for $x<1$ one has $\Lambda_2(T_x) =0$.
\end{itemize}
\end{prop}

The above properties allows one to give an upper bound for the Rayleigh quotients of $G_h$ by those of
a one-dimensional operator on $(0,h^k)$. Namely, denote by $K_h$ the self-adjoint operator in $L^2(0,h^k)$ given by
\begin{equation}
    \label{eqkh}
K_h[f,f]=\int_{0}^{h^k} \big( h^2 f'(x)^2+2x^p f(x)^2 \big)\dd x,
\quad
\cQ(K_h)=H^1_0(0,h^k).
\end{equation}
\begin{lemma}\label{lem9a}
There exists $a_0>0$ such that
\[
\Lambda_n(G_h)\le -1+\Lambda_n(K_h)+a_0(h^{2+2k(p-1)}+h^{2kp})
\text{ for all $h>0$ and $n\in\NN$.}
\]
\end{lemma}

\begin{proof}
If $f\in H^1_0(0,h^k)$, then for the function $u\in H^1_0(\Omega'_h)$ defined by
\[
u(x_1,x_2)=f(x_1)\Psi_{x_1^p}(x_2)
\]
we have $\|f\|_{L^2(0,h^k)}=\|u\|_{L^2(\Omega'_h)}$ and
\[
\iint_{\Omega'_h} (\partial_2 u)^2\dd x-\int_0^{h^k}\big(u(s,s^p)^2+u(s,-s^p)^2\big)\dd s\\
=\int_0^{h_k} \sigma(x_1^p)f(x_1)^2\dd x_1.
\]
The $L^2$-normalization of $\Psi_{x_1^p}$ implies
\[
\int_{\RR} \Psi_{x_1^p} (x_2) \,\partial_{x_1} \Psi_{x_1^p}(x_2)\dd x_2=\dfrac{1}{2}\,\partial_{x_1} \|\Psi_{x_1^p}\|_{L^2(\RR)}^2=0,
\]
hence,
\begin{align*}
\iint_{\Omega'_h} (\partial_1 u)^2\dd x&=\int_{0}^{h^k} \int_{\RR} \Big[ f'(x_1)^2 \Psi_{x_1^p}(x_2)^2\\
&\qquad + 2f(x_1)f'(x_1)\Psi_{x_1^p} (x_2) \partial_{x_1} \Psi_{x_1^p}(x_2)\\
&\qquad +f(x_1)^2 (\partial_{x_1} \Psi_{x_1^p}(x_2))^2\Big]\dd x_2\, \dd x_1\\
&=\int_0^{h_k} \big( f'(x_1)^2 + w(x_1) f(x_1)^2\big)\dd x_1, 
\end{align*}
where we denote $w(x_1):=\big\|\partial_{x_1} \Psi_{x_1}\big\|^2_{L^2(\RR)}\equiv p^2x_1^{2(p-1)}\big\|(\partial_z \Psi_z)_{z=x_1^p}\big\|^2_{L^2(\RR)}$, and 
\[
G_h[u,u]= \int_0^{h^k} \Big( h^2 f'(x_1)^2 +\big[\sigma(x_1^p)+h^2w(x_1)\big]\,f(x_1)^2\Big)\dd x_1
\]
Due to Proposition~\ref{prop1d}(c,d) for a sufficiently large $a_0>0$ one can estimate
\[
p^2\big\|(\partial_z \Psi_z)_{z=x_1^p}\big\|^2_{L^2(\RR)}\le a_0,
\quad
\sigma(x_1)\le -1+2x^p_1+a_0 h^{2kp},
\quad x_1\in(0,h^k),
\]
and then
\begin{multline*}
G_h[u,u]\le -\|f\|^2_{L^2(0,h^k)}+\int_{0}^{h^k} \big( h^2 f'(x_1)^2+2x^p f(x_1)^2 \big)\dd x_1\\
 + a_0(h^{2+2k(p-1)}+h^{2kp})\|f\|^2_{L^2(0,h^k)}.
\end{multline*}
Therefore, the linear operator $J:\cQ(K_h)\ni f\mapsto u\in \cQ(G_h)$
satisfies, for all $f \in \cQ(K_h)$, the equality $\|Jf\|_{L^2(\Omega'_h)}=\|f\|_{L^2(0,h^k)}$
and the inequality
\[
G_h[Jf,Jf]\le -\|f\|^2_{L^2(0,h^k)}+K_h[f,f]+a_0 (h^{2+2k(p-1)}+h^{2kp})\|f\|^2_{L^2(0,h^k)},
\]
which implies the claim by the min-max principle.
\end{proof}

\subsection{Analysis of the effective operator}\label{seckh}
Now we are reduced to the study of the eigenvalues of $K_h$ for small $h>0$. We will show that the princpal term of their asymptotics is determined by the eigenvalues of the model operator $A$.

For $\mu>0$, we introduce first two auxiliary operators $C_{N/D}^\mu$, which are the self-adjoint operators in $L^2(0,\mu)$ given by 
\begin{gather*}
C_{N/D}^\mu[f,f] = \int_0^\mu \big(f'(x)^2 + x^p f(x)^2 \big)\dd x,\\
\cQ(C_N^\mu)=\big\{ f\in H^1(0,\mu):\, f(0)=0\big\},
\quad
\cQ(C_D^\mu)=H^1_0(0,\mu).
\end{gather*}
An elementary scaling argument gives the following result:
\begin{lemma}\label{lem6a}
For any $n\in\NN$ and $h>0$ one has
\[
\Lambda_n(K_h) = 2^{\frac{2}{2+p}}h^{\frac{2p}{2+p}} \Lambda_n (C_{N/D}^\mu), \quad\mu:=2^{\frac{1}{2+p}}h^{k-\frac{2}{2+p}}.
\]
\end{lemma}
Remark that if $k<\tfrac{2}{2+p}$ then in the above representation one has $\mu\to+\infty$ as $h\to 0^+$.
Let us now study the behavior of the eigenvalues of $C^\mu_{N/D}$ for large $\mu>0$.

\begin{lemma}\label{lem7}
Let $n\in\NN$ be fixed, then $\Lambda_n(C_{N/D}^\mu) = \Lambda_n(A) + \cO(\mu^{-2})$ for $\mu\to+\infty$.
\end{lemma}

\begin{proof}
Directly by the min-max principle, for any $\mu>0$ one has the inequality
\begin{equation}
 \label{peq00}
 \Lambda_n(A)\le \Lambda_n(C_D^\mu).
\end{equation}
Furthermore, consider the self-adjoint operator $D_\mu$ in $L^2(\mu,+\infty)$  given by
\begin{gather*}
D_\mu[f,f] = \int_\mu^\infty \big( f'(x)^2 +x^p f(x)^2\big)\dd x,\\
\cQ(D_\mu)=\big\{f \in H^1(\mu,+\infty): \; x^{\frac{p}{2}}f \in L^2(\mu,+\infty)\big\},
\end{gather*}
then one clearly has $\Lambda_n(A) \geq \Lambda_n(C_N^\mu \oplus D_\mu)$ for any $\mu>0$.
The left-hand side is independent of $\mu$, while $D_{\mu}\ge \mu^p\to+\infty$ as $\mu\to+\infty$.
Therefore, there exists $\mu_n>0$ such that
\begin{equation}
  \label{peqA}
\Lambda_n(A) \ge \Lambda_n(C_N^\mu)  \text{ for } \mu\ge \mu_n.
\end{equation}

Now let $\chi_1,\chi_2\in C^\infty(\RR)$ such that
\[
\chi_1^2+\chi_2^2=1, \quad \chi_1(t) = 1 \text{ for } t\le \tfrac{1}{2}, \quad \chi_1(t) = 0  \text{ for } t\geq\tfrac{3}{4},
\]
and denote $\chi_{j,\mu} := \chi_j(\cdot/\mu)$.  Consider the self-adjoint operator $D'_\mu$ in $L^2(\frac{\mu}{2},\mu)$ given by
\[
D'_\mu[f,f] = \int_\frac{\mu}{2}^\mu \big( f'(x)^2 +x^p f(x)^2\big)\,\dd x,\quad
\cQ(D'_\mu)=H^1\big(\tfrac{\mu}{2},\mu\big).
\]
Then a direct computation shows that for any $f\in \cQ(C^\mu_N)$ one has, with $K:=\big\|(\chi'_1)^2 + (\chi'_2)^2\big\|_\infty$,
\begin{align*}
C^\mu_N[f,f] &= C^\mu_N[\chi_{1,\mu}f,\chi_{1,\mu}f] + C^\mu_N[\chi_{2,\mu}f,\chi_{2,\mu}f]\\
&\qquad - \int_0^\mu \big((\chi'_{1,\mu})^2 + (\chi'_{2,\mu})^2\big) f^2\dd x\\
&\geq C^\mu_N[\chi_{1,\mu}f,\chi_{1,\mu}f] + C^\mu_N[\chi_{2,\mu}f,\chi_{2,\mu}f] - K\mu^{-2} \|f\|^2_{L^2(0,\mu)}\\
&=C_D^\mu[\chi_{1,\mu}f,\chi_{1,\mu}f]+D'_\mu[\chi_{2,\mu}f,\chi_{2,\mu}f]- K\mu^{-2} \|f\|^2_{L^2(0,\mu)},\\
&=(C_D^\mu\,\oplus\, D'_\mu)[Jf,Jf]- K\mu^{-2} \|f\|^2_{L^2(0,\mu)},\\
Jf&:=(\chi_{1,\mu}f,\chi_{2,\mu}f),
\end{align*}
which implies $\Lambda_n(C^\mu_N) \geq \Lambda_n(C_D^\mu \oplus D'_\mu) - K\mu^{-2}$ for any $\mu>0$.
By \eqref{peqA}, for $\mu\to+\infty$ the left-hand side of the last inequality remains bounded, while
$D'_\mu\ge \mu^p2^{-p}\to+\infty$. Therefore, the value of $\mu_n$ in~\eqref{peqA}
can be assumed such that, in addition,
\begin{equation}
 \label{peqB}
\Lambda_n(C^\mu_N) \geq \Lambda_n(C_D^\mu) - K\mu^{-2} \text{ for any $\mu\ge \mu_n$.}
\end{equation}
By putting together the above estimates, for $\mu\ge \mu_n$ we obtain
\[
\Lambda_n(C_D^\mu) - K/\mu^2
\stackrel{\eqref{peqB}}{\le}
\Lambda_n(C_N^\mu)\stackrel{\eqref{peqA}}{\le}
\Lambda_n(A)\stackrel{\eqref{peq00}}{\le} \Lambda_n(C^\mu_D),
\]
which implies first $\Lambda_n(C^\mu_D)=\Lambda_n(A)+\cO(\mu^{-2})$ and then
$\Lambda_n(C_N^\mu)=\Lambda_n(C^\mu_D)+\cO(\mu^{-2})=\Lambda_n(A)+\cO(\mu^{-2})$.
\end{proof}

By combining Lemma~\ref{lem6a} with Lemma~\ref{lem7} we arrive at
\begin{lemma}\label{lem7b}
For any $n\in\NN$ and $k\in(0,\tfrac{2}{2+p})$ there holds
\[
\Lambda_n(K_h)=2^{\frac{2}{2+p}}h^{\frac{2p}{2+p}}\,\Lambda_n(A)+\cO(h^{2-2k})
\text{ as $h\to0^+$.}
\]
\end{lemma}

\subsection{Proof of the upper eigenvalue bound}\label{sec-upper}

The substitution of the asymptotics of Lemma~\ref{lem7b} (passage from  $K_h$ to $A$)
into Lemma~\ref{lem9a} (passage from $G_h$ to $K_h$)
shows that for every fixed $n\in\NN$ and $k\in(0,\tfrac{2}{2+p})$ there holds
\[
\Lambda_n(G_h)\le -1+2^{\frac{2}{2+p}}h^{\frac{2p}{2+p}}\,\Lambda_n(A)+\cO(h^{2+2k(p-1)}+h^{2kp}+h^{2-2k})
\]
as $h\to0^+$. For $k\in(0,\tfrac{2}{2+p})$ one has
\begin{gather*}
2+2k(p-1)=2kp+2(1-k)\ge 2kp,\\
\cO(h^{2+2k(p-1)}+h^{2kp}+h^{2-2k})=\cO(h^{2kp}+h^{2-2k}).
\end{gather*}
Taking $k:=\frac{1}{1+p}\in(0,\frac{2}{2+p})$ and then applying Lemma~\ref{lemx2} we see that
for any $\varepsilon>0$ and  $n\in\NN$ there holds, as $h\to 0^+$,
\begin{equation}   \label{upbd}
\Lambda_n(F_{h,\varepsilon h^{\frac{1}{1-p}}})\le\Lambda_n(G_h)\le -1+2^{\frac{2}{2+p}}h^{\frac{2p}{2+p}}\,\Lambda_n(A)+\cO(h^{\frac{2p}{1+p}}) < -\tfrac{1}{2}.
\end{equation}
It follows that the assumption~\eqref{asc1} is satisfied \emph{for any}
$\varepsilon>0$ and $n\in\NN$, which gives a stronger version of Lemma~\ref{lemyy}:
\begin{lemma}\label{lemeps}
For any $n\in\NN$ and $\varepsilon>0$ there holds
\begin{equation}
  \label{lnh}
\Lambda_n(H_\alpha)=\alpha^2\Lambda_n(F_{h,\varepsilon h^{\frac{1}{1-p}}})+\cO(1) \text{ for $h:=\alpha^\frac{1-p}{p}$
and $\alpha\to+\infty$.}
\end{equation}
\end{lemma}
Applying again~\eqref{upbd} to the right-hand side of \eqref{lnh} one arrives at
\begin{align*}
\Lambda_n(H_\alpha)&\le -\alpha^2+2^{\frac{2}{2+p}}\Lambda_n(A)\alpha^{\frac{6}{2+p}}+\cO(\alpha^{\frac{4}{1+p}})\\
&\equiv -\alpha^2+2^{\frac{2}{2+p}}\Lambda_n(A)\alpha^{\frac{6}{2+p}}+\cO(\alpha^{\frac{6}{2+p}-\eta}), \quad
\alpha\to+\infty.
\end{align*}
where $\eta:=\frac{6}{2+p}-\frac{4}{1+p}=\frac{2(p-1)}{(p+1)(p+2)}>0$. As the upper bound obtained
for $\Lambda_n(H_\alpha)$ is strictly negative for large $\alpha$, it lies below the essential spectrum of $H_\alpha$,
and it follows by the min-max principle that $\Lambda_n(H_\alpha)$ is the $n$th eigenvalue of $H_\alpha$.

\section{\bf Lower bound}\label{seclow}

\subsection{Reduction to a smaller half-plane}

Now we need to obtain a lower bound for the eigenvalues of $F_{h,\varepsilon h^{\frac{1}{1-p}}}$ with a suitably chosen $\varepsilon>0$.
Recall that
\begin{multline*}
F_{h,\varepsilon h^{\frac{1}{1-p}}}[u,u]=\iint_{\Omega_{\varepsilon h^{\frac{1}{1-p}}}} \big(h^2(\partial_1 u)^2+ (\partial_2 u)^2\big)\dd x\\
-\int_0^{\varepsilon h^{\frac{1}{1-p}}} \sqrt{1+p^2h^2s^{2(p-1)}} \big(u(s,s^p)^2+u(s,-s^p)^2\big)\dd s.
\end{multline*}
Let $k>0$, to be chosen later, and $h>0$ sufficiently small to have $h^k<\varepsilon h^{\frac{1}{1-p}}$.
Let $R_h$ be the self-adjoint operator in $L^2(\Omega_{h^k})$ given by
\begin{align*}
R_h[u,u]&=
\iint_{\Omega_{h^k}} \big(h^2(\partial_1 u)^2+ (\partial_2 u)^2\big)\dd x\\
&\quad-\int_0^{h^k} \sqrt{1+p^2h^{2+2k(p-1)}} \big(u(s,s^p)^2+u(s,-s^p)^2\big)\dd s,\\
\cQ(R_h)&=H^1(\Omega_{h^k}).
\end{align*}

\begin{lemma}\label{lem111}
Let $k\in\big(0,\frac{2}{2+p}\big)$. There exists $\varepsilon_1>0$
such that for any $\varepsilon\in(0,\varepsilon_1)$ and any $n\in\NN$ there holds
\[
\Lambda_n(F_{h,\varepsilon h^{\frac{1}{1-p}}})\ge \Lambda_n(R_h) \text{ as $h\to0^+$.}
\]
\end{lemma}

For the proof of Lemma~\ref{lem111} we need an auxiliary one-dimensional operator, which will also plays a role on later steps.
For $x>0$ and $\beta>0$ we denote by $T_{x,\beta}$ the self-adjoint operator in $L^2(\RR)$
given by
\begin{equation*}
T_{x,\beta}[f,f] = \int_{\RR} f'(y)^2\dd y -\beta\big(f(x)^2+ f(-x)^2\big), \quad \cQ(T_{x,\beta})=H^1(\RR),
\end{equation*}
which is closely related to the operator $T_x$ from \eqref{eqtx} and Proposition~\ref{prop1d}: a simple
scaling argument shows that
$T_{x,\beta}$ is unitarily equivalent to $\beta^2 T_{\beta x}$ and
$\Lambda_n(T_{x,\beta})=\beta^2\Lambda_n(T_{\beta x})$ for any $n\in\NN$.
In particular,
\[
\Lambda_1(T_{x,\beta})=\beta^2 \sigma(\beta x).
\]

\begin{proof}[\bf Proof of Lemma~\ref{lem111}]
By considering separately the integrals for $x_1<h^k$ and $x_1>h^k$ we arrive at
$F_{h,\varepsilon h^{\frac{1}{1-p}}}[u,u]=I_1+I_2$ with
\begin{align*}
I_1&=\iint_{\Omega_{h^k}} \big(h^2(\partial_1 u)^2+ (\partial_2 u)^2\big)\dd x\\
&\quad-\int_0^{h^k} \sqrt{1+p^2h^2 x_1^{2(p-1)}} \big(u(s,s^p)^2+u(s,-s^p)^2\big)\dd s,\\
I_2&=\int_{h^k}^{\varepsilon h^{\frac{1}{1-p}}} \bigg[ \int_\RR\big( h^2 (\partial_1 u)^2 + (\partial_2 u)^2\big)\dd x_2\\
&\quad  -\sqrt{1+p^2h^2 x_1^{2(p-1)}} \big(u(x_1,x_1^p)^2+u(x_1,-x_1^p)^2\big)\bigg]\dd x_1,
\end{align*}
and one has obviously $I_1\ge R_h[u_1,u_1]$ with $u_1:=u|_{\Omega_{h^k}}$.

Now one needs a lower bound for $I_2$. First, by dropping the non-negative  term $(\partial_1 u)^2$ and using the above one-dimensional operator operator $T_{x,\beta}$ we estimate
\[
I_2 \ge\int_{h^k}^{\varepsilon h^{\frac{1}{1-p}}}  \lambda(x_1,h) \int_\RR u(x_1,x_2)^2\dd x_2 \dd x_1,
\]
where we denoted
\begin{align*}
\lambda(x_1,h)&:=\Lambda_1\big(T_{x_1^p,\sqrt{1+p^2h^2 x_1^{2(p-1)}}}\big)\\
&\equiv \big(1+p^2h^2 x_1^{2(p-1)}\big)\sigma\big( \sqrt{1+p^2h^2 x_1^{2(p-1)}} \, x_1^p \big).
\end{align*}
To estimate $\lambda(x_1,h)$ from below let us pick $q\in(0,\frac{1}{p-1})$, then for small $h$ one has $h^k<h^{-q}<\varepsilon h^{\frac{1}{1-p}}$.

Consider first the values $x_1\in(h^k,h^{-q})$. Due to
\[
\sqrt{1+p^2h^2 x_1^{2(p-1)}}\, x_1^p> x_1^p>h^{kp},
\]
by Proposition~\ref{prop1d}(a,b) one obtains
\[
\sigma( h^{kp}) \le \sigma\big( \sqrt{1+p^2h^2 x_1^{2(p-1)}} \, x_1^p\big)<0.
\]
On the other hand, $1+p^2h^2 x_1^{2(p-1)}< 1+p^2h^{2-2q(p-1)}$, which together with the preceding estimate
gives
\[
\big(1+p^2h^2 x_1^{2(p-1)}\big)\sigma\big( \sqrt{1+p^2h^2 x_1^{2(p-1)}} \, x_1^p\big)\ge (1+p^2h^{2-2q(p-1)})\sigma( h^{kp}).
\]
Using Proposition~\ref{prop1d}(c) to estimate $\sigma(h^{kp})$, for small $h>0$ we arrive at 
\[
\lambda(x_1,h)\ge(1+p^2h^{2-2q(p-1)})\Big(-1+\frac{3}{2}\,h^{kp}\Big)\ge-1+\tfrac{3}{2}h^{kp} -p^2h^{2-2q(p-1)}.
\]
As $k$ and $q$ were rather arbitrary so far, we may assume that 
\[
kp<2,\quad 0<q<\tfrac{2-kp}{2(p-1)}\equiv \tfrac{1-\tfrac{kp}{2}}{p-1}<\tfrac{1}{p-1},
\]
then $kp<2-2q(p-1)$ and $h^{2-2q(p-1)}=o(h^{kp})$. Therefore,
\begin{equation}
   \label{lll1}
\lambda(x_1,h)\ge -1+h^{kp} \text{ for $x_1\in(h^k,h^{-q})$ and $h\to 0^+$.}
\end{equation}

Keeping the above value of $q$ consider now $x_1\in \big(h^{-q},\varepsilon h^{\frac{1}{1-p}}\big)$.
We have first
\[
\sqrt{1+p^2h^2 x_1^{2(p-1)}}\, x_1^p> x_1^p>h^{-pq}
\]
and then, by Proposition~\ref{prop1d}(a,b),
\[
\sigma( h^{-pq}) \le \sigma\big( \sqrt{1+p^2h^2 x_1^{2(p-1)}}\, x_1^p\big)<0.
\]
In addition, $1+p^2h^2 x_1^{2(p-1)}\le 1+p^2 \varepsilon^{2(p-1)}$, and $\sigma(h^{-pq})<0$, therefore,
\[
\lambda(x_1,h)\ge (1+p^2 \varepsilon^{2(p-1)}) \sigma( h^{-pq}).
\]
In view of Proposition~\ref{prop1d}(b,c), one can choose $\delta>0$ sufficiently small such that $\sigma( h^{-pq})\ge -1+2\delta$
for small $h>0$. In addition, we may take $\varepsilon_1>0$ sufficiently small to have $p^2 \varepsilon_1^{2(p-1)}<\delta$,
then for any $\varepsilon\in(0,\varepsilon_1)$ one $\lambda(x_1,h)\ge (1+\delta)(-1+2\delta)\ge -1+\delta$ for small $h$.
By combining with \eqref{lll1} we see that $\lambda(x_1,h)\ge -1+h^{kp}$ for all $x_1\in(h^k,\varepsilon h^{\frac{1}{1-p}})$ if $h$ is sufficiently small, and then
\[
I_2\ge (-1+h^{kp})\int_{h^k}^{\varepsilon h^{\frac{1}{1-p}}} \int_\RR u(x_1,x_2)^2\dd x_2\dd x_1.
\]

We summarize the above estimates as follows: there exist $\varepsilon\in(0,\varepsilon_1)$ and $h_1>0$
such that for all $h\in(0,h_1)$ and $u\in\cQ(F_{h,\varepsilon h^{\frac{1}{1-p}}})$
there holds
\begin{gather*}
F_{h,\varepsilon h^{\frac{1}{1-p}}}\ge R_n[u_1,u_1]+(-1+h^{kp})\|u_2\|^2_{L^2(\Omega_{\varepsilon h^{\frac{1}{1-p}}}\setminus {\Omega_{h^k}})},\\
u_1:=u|_{\Omega_{h^k}}, \quad u_2:=u|_{\Omega_{\varepsilon h^{\frac{1}{1-p}}}\setminus {\Omega_{h^k}}},
\end{gather*}
and then for any fixed $n\in\NN$ and small $h$ one has
\begin{equation}
   \label{lll2}
\Lambda_n(F_{h,\varepsilon h^{\frac{1}{1-p}}})\ge \min\big\{\Lambda_n(R_h), -1+h^{kp}\big\}.
\end{equation}
The min-max principle shows that $\Lambda_n(R_h)\le \Lambda_n(G_h)$ for the operator $G_h$ from Subsection~\ref{secgh},
and the estimate \eqref{upbd} for $\Lambda_n(G_h)$ yields $\Lambda_n(R_h)\le -1+\cO(h^{\frac{2p}{2+p}})$.
For $k\in(0,\frac{2}{2+p})$ one has $h^{\frac{2p}{2+p}}=o(h^{kp})$ and then $\Lambda_n(R_h)<-1+h^{kp}$.
The substitution into \eqref{lll2} concludes the proof.
\end{proof}

\subsection{Reduction to a one-dimensional problem}

In the present section we will provide a lower bound for the eigenvalues of $\Lambda_n(R_h)$ in terms of a one-dimensional
operator. Namely, consider the function
\[
V:x\mapsto \begin{cases}1, & x<0, \\ 2x^p, & x>0\end{cases},
\]
and the operator $Z_h$ in $L^2(-\infty,h^k)$ given by $Z_h f=-h^2f''+Vf$
with Neumann condition at the right end, $f'(h^k)=0$, i.e. 
\begin{gather*}
Z_h[f,f] = h^2\int_{-\infty}^{h^k} f'(x)^2 \dd x +  \int_{-\infty}^0 f(x)^2 \dd x + 2\int_0^{h^k} x^p f(x)^2 \dd x
\end{gather*}
with $\cQ(Z_h)=H^1(-\infty,h^k)$.

\begin{lemma}\label{lem15}
For any $n\in\NN$, $k\in(0,\frac{2}{2+p})$ and $s>0$ there holds
\[
\Lambda_n(R_h)\ge -1+\Lambda_n(Z_{h_0})+\cO(h^{2+2k(p-1)-s}+h^{2kp}), \quad h\to 0^+,
\]
where we denote
\[
h_0:=h\sqrt{1-h^s}.
\]
\end{lemma}
The proof will occupy the rest of the subsection.

It will be convenient to use the one-dimensional operator
\[
L_{x_1,h}:=T_{x_1^p,\sqrt{1+p^2h^{2+2k(p-1)}}},
\]
its first eigenvalue
\begin{align*}
\kappa(x_1,h)&:=\Lambda_1(L_{x_1,h})\equiv\Lambda_1\big(T_{x_1^p,\sqrt{1+p^2h^{2+2k(p-1)}}}\big)\\
&\equiv \big(1+p^2h^{2+2k(p-1)}\big)\sigma\big( \sqrt{1+p^2h^{2+2k(p-1)}}\,x_1^p\big),
\end{align*}
and the associated eigenfunction $\Phi_{x_1,h}$ chosen positive and normalized by $\|\Phi_{x_1,h}\|_{L^2(\RR)}=1$. In terms of the first eigenfunction $\Psi_x$ of $T_x$
one has clearly
\[
\Phi_{x_1,h}(t)=\sqrt[4]{1+p^2h^{2+2k(p-1)}}\,\Psi_{\sqrt{1+p^2h^{2+2k(p-1)}}\,x_1^p}(\sqrt{1+p^2h^{2+2k(p-1)}}\, t).
\]
Due to Proposition~\ref{prop1d} for any $h>0$ the function $x_1\mapsto \Phi_{x_1,h}$ admits a finite limit $\Phi_{0,h}$
at $x_1=0^+$,
so we define
\[
\Hat \Phi_{x_1,h}=\begin{cases}\Phi_{x_1,h}, & x_1>0,\\ \Phi_{0,h}, & x_1<0.\end{cases}
\]
Consider the following closed subspace $\cG$ of $L^2(\Omega_{h^k})$,
\[
\cG:=\big\{ (x_1,x_2)\mapsto f(x_1)\Hat\Phi_{x_1,h}(x_2):\, f\in L^2(-\infty,h^k)\big\},
\]
and denote by $\Pi$ the orthogonal projector on $\cG$ in $L^2(\Omega_{h^k})$, then the operator $\Pi^\perp:=1-\Pi$ is the orthogonal projector on $\cG^\perp$.
One easily checks that for $u\in L^2(\Omega_{h^k})$ there holds
\begin{gather*}
(\Pi u)(x_1,x_2)=f(x_1)\Hat \Phi_{x_1,h}(x_2) \text{ with } f(x_1)=\int_{\RR} \Hat\Phi_{x_1,h}(x_2)u(x_1,x_2)\dd x_2,\\
\|\Pi u\|^2_{L^2(\Omega_{h^k})}=\|f\|^2_{L^2(-\infty,h^k)},
\end{gather*}
and that for $u\in \cQ(R_h)$ one has $f\in H^1(-\infty,h^k)$. We keep this correspondence between $u$ and $f$ for subsequent computations.
Recall that
\begin{align*}
R_h[u,u]&=
\iint_{\Omega_{h^k}} \big(h^2(\partial_1 u)^2+ (\partial_2 u)^2\big)\dd x\\
&\quad-\int_0^{h^k} \sqrt{1+p^2h^{2+2k(p-1)}} \big(u(s,s^p)^2+u(s,-s^p)^2\big)\dd s.
\end{align*}
Using the spectral theorem for the above operator $L_{x_1,h}$ we obtain
\begin{align*}
I&:=\iint_{\Omega_{h^k}}  (\partial_2 u)^2\dd x\\
&\qquad-\int_0^{h^k} \sqrt{1+p^2h^{2+2k(p-1)}} \big(u(s,s^p)^2+u(s,-s^p)^2\big)\dd s\\
&\ge \iint_{\Omega_{h^k}\cap\{x_1>0\}}  (\partial_2 u)^2\dd x\\
&\qquad-\int_0^{h^k} \sqrt{1+p^2h^{2+2k(p-1)}} \big(u(s,s^p)^2+u(s,-s^p)^2\big)\dd s\\
&=\int_0^{h^k} \Big[\int_{\RR} \partial_2 u (x_1,x_2)^2\dd x_2\\
&\qquad -\sqrt{1+p^2h^{2+2k(p-1)}}\, \big(u(x_1,x_1^p)^2+u(x_1,-x_1^p)^2\big)\Big]\dd x_1\\
&\ge \int_0^{h^k}\Big(\Lambda_1(L_{x_1,h}) \|\Pi u(x_1,\cdot)\|^2_{L^2(\RR)}+\Lambda_2(L_{x_1,h})\|\Pi^\perp u(x_1,\cdot)\|^2_{L^2(\RR)}\Big)\dd x_1.
\end{align*}
Assuming that $h$ is small, by Proposition~\ref{prop1d}(e) one obtains, for any $x_1\in(0,h^k)$,
\[
\Lambda_2(L_{x_1,h})=\big(1+p^2h^{2+2k(p-1)}\big)\Lambda_2\big( T_{\sqrt{1+p^2h^{2+2k(p-1)}}\,x_1^p}\big)=0,
\]
which gives
\[
I\ge \int_{0}^{h^k} \kappa(x_1,h) \|\Pi u(x_1,\cdot)\|^2_{L^2(\RR)}\dd x_1\equiv \int_{0}^{h^k} \kappa(x_1,h) f(x_1)^2\dd x_1.
\]
Hence, if $h$ is sufficiently small, for any $u\in \cQ(R_h)$ we have
\begin{equation}
   \label{eqrh1}
R_h[u,u]\ge h^2\iint_{\Omega_{h^k}} (\partial_1 u)^2\dd x+\int_{0}^{h^k}  \kappa(x_1,h) f(x_1)^2\dd x_1.
\end{equation}
To obtain a lower bound for the first summand on the right-hand side we start with
\begin{align*}
\Pi\partial_1u(x_1,x_2)&=\int_\RR \Hat \Phi_{x_1,h}(t)\partial_1 u (x_1,t)\dd t \,\Hat \Phi_{x_1,h}(x_2),\\
\partial_1 \Pi u(x_1,x_2)&=\dfrac{\partial}{\partial x_1} \Big(\int_\RR \Hat \Phi_{x_1,h}(t) u (x_1,t)\dd t \,\Hat \Phi_{x_1,h}(x_2)\Big)\\
&=\int_\RR \Hat \Phi_{x_1,h}(t)\partial_1 u (x_1,t)\dd t\,\Hat \Phi_{x_1,h}(x_2)\\
&\qquad+\int_\RR (\partial_{x_1} \Hat \Phi_{x_1,h})(t) u (x_1,t)\dd t\,\Hat \Phi_{x_1,h}(x_2)\\
&\qquad+\int_\RR \Hat \Phi_{x_1,h}(t) u (x_1,t)\dd t\,(\partial_{x_1}\Hat\Phi_{x_1,h})(x_2).
\end{align*}
Therefore, using $(a+b)^2\le 2a^2+2b^2$ and Cauchy-Schwarz inequality,
\begin{multline*}
\big|(\Pi\partial_1 -\partial_1\Pi)u(x_1,x_2)\big|^2\\
\begin{aligned}
&=\bigg|\int_\RR (\partial_{x_1} \Hat \Phi_{x_1,h})(t) u (x_1,t)\dd t\,\Hat \Phi_{x_1,h}(x_2)\\
&\qquad+\int_\RR \Hat \Phi_{x_1,h}(t) u (x_1,t)\dd t\,(\partial_{x_1}\Hat\Phi_{x_1,h})(x_2)\bigg|^2\\
&\le 2 \|\partial_{x_1} \Hat \Phi_{x_1,h}\|^2_{L^2(\RR)} \|u (x_1,\cdot)\|^2_{L^2(\RR)} \Hat\Phi_{x_1,h}(x_2)^2\\
&\qquad +2\|\Hat \Phi_{x_1,h}\|^2_{L^2(\RR)} \|u (x_1,\cdot)\|^2_{L^2(\RR)} (\partial_{x_1}\Hat \Phi_{x_1,h})(x_2)^2.
\end{aligned}
\end{multline*}
We further recall that $\|\Hat \Phi_{x_1,h}\|^2_{L^2(\RR)}=1$ for all $x_1$
and that
\[
\partial_{x_1}\Hat \Phi_{x_1,h}=\begin{cases}
\partial_{x_1}\Phi_{x_1,h}, & x_1>0,\\
0, &x_1<0.
\end{cases}
\]
This gives
\begin{multline*}
\big\|(\Pi\partial_1 -\partial_1\Pi)u\big\|^2_{L^2(\Omega_{h^k})}\\
\begin{aligned}
&\le
2\int_0^{h^k} \|\partial_{x_1} \Phi_{x_1,h}\|^2_{L^2(\RR)} \|u (x_1,\cdot)\|^2_{L^2(\RR)} \Big(\int_\RR\Phi_{x_1,h}(x_2)^2\dd x_2\Big)\dd x_1\\
&\quad +2 \int_0^{h^k} \|u (x_1,\cdot)\|^2_{L^2(\RR)} \Big(\int_\RR(\partial_{x_1} \Phi_{x_1,h})(x_2)^2\dd x_2\Big) \dd x_1\\
&\le 4 \int_0^{h^k} w(x_1,h) \|u (x_1,\cdot)\|^2_{L^2(\RR)}\dd x_1, 
\end{aligned}
\end{multline*}
where we denoted
\[
w(x_1,h):=\|\partial_{x_1} \Phi_{x_1,h}\|^2_{L^2(\RR)}.
\]
With $\lambda:=\sqrt{1+p^2h^{2+2k(p-1)}}$ we have $\Phi_{x_1,h}(t)=\sqrt{\lambda} \Psi_{\lambda\,x_1^p}(\lambda\, t)$ and
\begin{align*}
w(x_1,h)&= \lambda^3 \int_\RR p^2x_1^{2(p-1)}(\partial_z\Psi_z)_{z=\lambda\,x_1^p}(\lambda\, t)^2\dd t\\
&=\lambda^2 p^2x_1^{2(p-1)}\int_\RR (\partial_z\Psi_z)_{z=\lambda\,x_1^p}(t)^2\dd t\\
&\le p^2(1+p^2h^{2+2k(p-1)})x_1^{2(p-1)} \sup_{z>0}\|\partial_z\Psi_z\|^2_{L^2(\RR)}
\end{align*}
Due to Proposition~\ref{prop1d}(d) the last factor on the right-hand side is finite, and for a suitable $b_0>0$ one obtains
$w(x_1,h)\le b_0 x_1^{2(p-1)}$,
and then
\begin{align*}
\big\|(\Pi\partial_1 -\partial_1\Pi)u\big\|^2_{L^2(\Omega_{h^k})}&\le 4 \int_0^{h^k} b_0 x_1^{2(p-1)} \|u (x_1,\cdot)\|^2_{L^2(\RR)}\dd x_1\\
&\le 4b_0 h^{2k(p-1)}\|u\|^2_{L^2(\Omega_{h^k})}.
\end{align*}
In addition, the function $(\Pi^\perp\partial_1 -\partial_1\Pi^\perp)u\equiv-(\Pi\partial_1 -\partial_1\Pi)u$ admits the same norm estimate. Using $(a+b)^2\ge (1-\delta)a^2-\delta^{-1}b^2$ for $a,b\in\RR$ and $\delta>0$ we estimate, with any $\delta>0$,
\begin{align*}
\big\|\partial_1 u\big\|_{L^2(\Omega_{h^k})}^2&=\big\|\Pi\partial_1 u\big\|_{L^2(\Omega_{h^k})}^2+\big\|\Pi^\perp\partial_1 u\big\|_{L^2(\Omega_{h^k})}^2\\
&=\big\|\partial_1 \Pi u+(\Pi\partial_1 -\partial_1\Pi)u\big\|_{L^2(\Omega_{h^k})}^2\\
&\qquad+\big\|\partial_1 \Pi^\perp u+(\Pi^\perp\partial_1 -\partial_1\Pi^\perp)\big\|_{L^2(\Omega_{h^k})}^2\\
&\ge (1-\delta)\big\|\partial_1 \Pi u\big\|_{L^2(\Omega_{h^k})}^2-\delta^{-1}\,\big\|(\Pi\partial_1 -\partial_1\Pi)u\big\|_{L^2(\Omega_{h^k})}^2\\
&\qquad +(1-\delta)\big\|\partial_1 \Pi^\perp u\big\|_{L^2(\Omega_{h^k})}^2\\
&\qquad -\delta^{-1}\,\big\|(\Pi^\perp\partial_1 -\partial_1\Pi^\perp)u\big\|_{L^2(\Omega_{h^k})}^2\\
&\ge(1-\delta)\big\|\partial_1 \Pi u\big\|_{L^2(\Omega_{h^k})}^2 -b \delta^{-1} h^{2k(p-1)}\,\|u\|_{L^2(\Omega_{h^k})}^2,
\end{align*}
where we took $b:=8b_0$. To estimate the term with $\partial_1 \Pi u$ we compute
\[
(\partial_1 \Pi)u(x_1,x_2)=f'(x_1)\Hat \Phi_{x_1,h}(x_2) + f(x_1) \partial_{x_1} \Hat \Phi_{x_1,h}(x_2)
\]
and remark that due to
\[
\int_\RR\Hat \Phi_{x_1,h}(x_2) \partial_{x_1} \Hat \Phi_{x_1,h}(x_2)\dd x_1=\dfrac{1}{2}\dfrac{\dd}{\dd x_1} \|\Hat \Phi_{x_1,h}\|^2_{L^2(\RR)}=0
\]
we have 
\begin{align*}
\|\partial_1 \Pi u\big\|_{L^2(\Omega_{h^k})}^2&=\int_{-\infty}^{h^k} f'(x_1)^2\int_\RR \Phi_{x_1,h}(x_2)^2\dd x_2\, \dd x_1\\
&\quad+
\int_{-\infty}^{h^k} f(x_1)^2 \int_\RR \big(\partial_{x_1} \Hat \Phi_{x_1,h}\big)(x_2)^2\dd x_2\, \dd x_1\\
&\ge\int_{-\infty}^{h^k} f'(x_1)^2\dd x_1.
\end{align*}
Therefore,
\[
\big\|\partial_1 u\big\|_{L^2(\Omega_{h^k})}^2\ge (1-\delta) \|f'\|^2_{L^2(-\infty,h^k)}-b \delta^{-1} h^{2k(p-1)}\,\|u\|_{L^2(\Omega_{h^k})}^2,
\]
and the substitution
into \eqref{eqrh1} gives
\begin{multline*}
R_h[u,u]+b \delta^{-1} h^{2+2k(p-1)}\,\|u\|_{L^2(\Omega_{h^k})}^2\\
\ge h^2(1-\delta) \int_{-\infty}^{h^k} f'(x_1)^2\dd x_1+\int_0^{h^k} \kappa(x_1,h) f(x_1)^2\dd x_1.
\end{multline*}
For what follows it is convenient to set $\delta:=h^s$ with $s>0$ to be chosen later,
then
\begin{multline}
   \label{rhs}
R_h[u,u]+b h^{2+2k(p-1)-s}\,\|u\|_{L^2(\Omega_{h^k})}^2\\
\ge h^2(1-h^s) \int_{-\infty}^{h^k} f'(x_1)^2\dd x_1+\int_0^{h^k} \kappa(x_1,h) f(x_1)^2\dd x_1.
\end{multline}
In view of Proposition~\ref{prop1d}(c) one can find a constants $a_0,a>0$ such that for small $h$ and $x_1\in(0,h^k)$ there holds
\begin{align*}
\kappa(x_1,h)&=\big(1+p^2h^{2+2k(p-1)}\big)\sigma\big( \sqrt{1+p^2h^{2+2k(p-1)}}\,x_1^p\big)\\
&\ge \big(1+p^2h^{2+2k(p-1)}\big)\big(-1+2 \sqrt{1+p^2h^{2+2k(p-1)}}\,x_1^p \\
&\qquad- a_0 (1+p^2h^{2+2k(p-1)})\,x_1^{2p}\big)\\
&\ge \big(1+p^2h^{2+2k(p-1)}\big)(-1+2x^p-2a_0 h^{2kp}\big)\\
&\ge -1+2x^p-a(h^{2+2k(p-1)}+h^{2kp}).
\end{align*}
Substituting this inequality into \eqref{rhs} and taking into account the inequality
$\|f\|^{2}_{L^2(0,h^k)}\equiv\|\Pi u\|_{L^2(\Omega_{h^k})}^2\le \|u\|_{L^2(\Omega_{h^k})}^2$
we obtain, with some constant $B>0$,
\begin{multline*}
R_h[u,u]+B (h^{2+2k(p-1)-s}+h^{2+2k(p-1)}+h^{2kp})\,\|u\|_{L^2(\Omega_{h^k})}^2\\
\ge h^2(1-h^s) \int_{-\infty}^{h^k} f'(x_1)^2\dd x_1+\int_0^{h^k} (-1+2x_1^p) f(x_1)^2\big)\dd x_1.
\end{multline*}
For $s>0$ we clearly have $h^{2+2k(p-1)}=o(h^{2+2k(p-1)-s})$, hence, with some $B'>B$,
\begin{multline}
   \label{rhs2}
R_h[u,u]+B' (h^{2+2k(p-1)-s}+h^{2kp})\,\|u\|_{L^2(\Omega_{h^k})}^2\\
\ge h^2(1-h^s) \int_{-\infty}^{h^k} f'(x_1)^2\dd x_1+\int_0^{h^k} (-1+2x_1^p) f(x_1)^2\big)\dd x_1\\
\equiv (-\one+Z_{h_0})[f,f].
\end{multline}
Consider now the isometric map
\[
J:L^2(\Omega_{h^k})\ni u\mapsto (f,\Pi^\perp u)\in L^2(-\infty,h^k)\oplus \cG^\perp,
\]
then the estimate \eqref{rhs2} can be rewritten as
\[
\big(R_h+B' (h^{2+2k(p-1)-s}+h^{2kp})\big)[u,u]\ge \big((-\one+Z_{h_0})\,\oplus\, 0 \big)[Ju,Ju].
\]
As this holds for all $u\in\cQ(R_h)$, the min-max principle shows that for any fixed $n\in\NN$ one has, as $h\to 0^+$,
\begin{align*}
\Lambda(R_h)+B' (h^{2+2k(p-1)-s}+h^{2kp})&\ge \Lambda_n\big((-\one+Z_{h_0})\,\oplus\, 0 \big)\\
=\min\big\{\Lambda_n(-\one+Z_{h_0}),0\big\}&=-1+\min\big\{\Lambda_n(Z_{h_0}),1\big\}.
\end{align*}
The min-max principle also shows that for any $n\in\NN$ and $h>0$ one has $\Lambda_n(Z_h)\le\Lambda_n(K_h)$,
where the operator $K_h$ was defined in \eqref{eqkh}, and it was shown in Lemma~\ref{lem6a}
that $\Lambda_n(K_h)=o(1)$ for small $h$.
It follows that $\Lambda_n(Z_{h_0})=o(1)$, and then $\min\big\{\Lambda_n(Z_{h_0}),1\big\}=\Lambda_n(Z_{h_0})$.
This gives finally $\Lambda(R_h)\ge -1+\Lambda_n(Z_{h_0})+\cO(h^{2+2k(p-1)-s}+h^{2kp})$. This proves Lemma~\ref{lem15}.

\subsection{One-dimensional analysis}

Now we need a more precise analysis of $Z_h$ for small $h$. We are going to prove the following result,
whose proof will occupy the rest of the subsection:

\begin{lemma}\label{lemzh}
Let $0< k < \frac{2}{p+2}$, then for any $n\in\NN$ there holds
\[
E_n(Z_h) = 2^{\frac{2}{2+p}}E_n(A)\,h^{\frac{2p}{p+2}}  + \cO( h^\frac{5p}{2p+4} + h^{2-2k}) \text{ as }h\to 0^+.
\]
\end{lemma}
It appears more convenient to change the scale in order to work with large constants. Namely,
for $\lambda>0$ and $\mu > 0$ we introduce self-adjoint operators $B^{\mu,\lambda}$ in $L^2(-\infty,\mu)$
by
\begin{gather*}
B^{\mu,\lambda}[f,f] = \int_{-\infty}^\mu f'(x)^2 \dd x + \lambda \int_{-\infty}^0 f(x)^2 \dd x + \int_0^\mu x^p f(x)^2 \dd x,\\
\cQ(B^{\mu,\lambda})=H^1(-\infty,\mu).
\end{gather*}
An elementary scaling argument gives the following result:
\begin{lemma}\label{lem6}
For any $n\in\NN$ one has
$\Lambda_n(Z_h) = 2^{\frac{2}{2+p}}h^{\frac{2p}{2+p}} \Lambda_n ( B^{\lambda,\mu})$
with $\lambda=2^{\frac{2}{2+p}}\,h^{-\frac{2p}{2+p}}$ and $\mu=2^{\frac{1}{2+p}}h^{k-\frac{2}{2+p}}$.
\end{lemma}
In view of Lemma~\ref{lem6} the behavior of the eigenvalues of $Z_h$ for $h\to 0^+$ can be deduced from that
of the eigenvalues of $B^{\lambda,\mu}$
for $\lambda\to+\infty$ and $\mu\to+\infty$. The latter will be again approached using the auxiliary operators $C^\mu_{N/D}$
already studied in Subsection~\ref{seckh}.

\begin{lemma}\label{lem8}
For any $n\in\NN$ there exists $\lambda_n>0$ and $M_n>0$ such that
\begin{equation}
  \label{blbl}
\Lambda_j(C_N^\mu) - K \lambda^{-\frac{1}{4}}  \leq \Lambda_j(B^{\lambda,\mu}) \leq \Lambda_j(C_D^\mu).
\end{equation}
for all $(\lambda,\mu)\in(\lambda_n,+\infty)\times(1,+\infty)$.
\end{lemma}

\begin{proof}
Remark first that all operators $B^{\lambda,\mu}$ and $C_{N/D}^\mu$ are non-negative. For $\mu>1$ and $\lambda>0$ the min-max principle gives
\begin{equation} \label{ee1}
0 \leq \Lambda_n(B^{\mu,\lambda}) \leq \Lambda_n(C_D^\mu) \leq \Lambda_n(C_D^1),
\end{equation}
and it follows, in particular, that the eigenvalue $\Lambda_n(B^{\mu,\lambda})$ is uniformly bounded.
It remains to show the first inequality in \eqref{blbl}. As the participating operators act in different spaces, it will be
convenient to use Proposition~\ref{Comp}, and we remark that this proof scheme is inspired by the constructions of~\cite{post}.
Consider the linear map
\[
J: \cQ(B^{\lambda,\mu})\to \cQ(C_N^\mu),
\quad
(Jf)(x)=f(x)-f(0)e^{-x}, \quad x\in(0,\mu).
\]
For any $\varepsilon>0$ and $a,b\in\RR$ one has $(a+b)^2\ge (1-\varepsilon) a^2+ \varepsilon^{-1}b^2$.
Therefore, for any $f\in H^1(-\infty,\mu)$ and $\varepsilon>0$ one has
\begin{align*}
\| J f \|^2_{L^2(0,\mu)} &= \int_0^\mu \big( f(x) - f(0) e^{-x}\big)^2 \dd x \\
&\geq (1-\varepsilon) \int_0^\mu f(x)^2\dd x - \varepsilon^{-1} \int_0^\mu f(0)^2 e^{-2x} \dd x \\
&\geq (1-\varepsilon) \|f \|^2_{L^2(0,\mu)} - \varepsilon^{-1}\,f(0)^2,
\end{align*}
resulting in
\begin{equation}
   \label{ff1}
\|f \|^2_{L^2(-\infty,\mu)} - \| J f \|^2_{L^2(0,\mu)} \leq \varepsilon \|f \|^2_{L^2(0,\mu)} + \varepsilon^{-1}\,f(0)^2  + \|f\|^2_{L^2(-\infty,0)}.
\end{equation}
For any $\delta>0$ one can estimate
\[
f(0)^2=2\int_{-\infty}^0 f(x) f'(x)\dd x\le \delta\|f'\|^2_{L^2(-\infty,0)}+\delta^{-1}\|f\|^2_{L^2(-\infty,0)},
\]
and the substitution into \eqref{ff1} yields
\begin{multline*}
\|f \|^2_{L^2(-\infty,\mu)} - \| J f \|^2_{L^2(0,\mu)} \\
\begin{aligned}
&\le
\varepsilon \|f \|^2_{L^2(0,\mu)}+\delta \varepsilon^{-1}\|f'\|^2_{L^2(-\infty,0)} + \varepsilon^{-1}\delta^{-1}\,\|f\|^2_{L^2(-\infty,0)}+\|f\|^2_{L^2(-\infty,0)}\\
&=
\varepsilon \|f \|^2_{L^2(0,\mu)}+\delta \varepsilon^{-1}\|f'\|^2_{L^2(-\infty,0)} + (\varepsilon^{-1}\delta^{-1}\lambda^{-1}+\lambda^{-1})\,\lambda\|f\|^2_{L^2(-\infty,0)}.
\end{aligned}
\end{multline*}
We now set $\delta:=\lambda^{-\frac{1}{2}}$ and $\varepsilon:=\lambda^{-\frac{1}{4}}$, then for $\lambda>1$ we have
\begin{multline*}
\|f \|^2_{L^2(-\infty,\mu)} - \| J f \|^2_{L^2(0,\mu)}\\
\begin{aligned}
&\le \lambda^{-\frac{1}{4}} \|f \|^2_{L^2(0,\mu)}+\lambda^{-\frac{1}{4}}\|f'\|^2_{L^2(-\infty,0)}
+ (\lambda^{-\frac{1}{4}}+\lambda^{-1})\,\lambda \|f\|^2_{L^2(-\infty,0)}\\
&\le 2\lambda^{-\frac{1}{4}} \big(\|f \|^2_{L^2(0,\mu)}+\|f'\|^2_{L^2(-\infty,0)}+ \lambda \|f\|^2_{L^2(-\infty,0)}\big),
\end{aligned}
\end{multline*}
and it follows that
\begin{equation} \label{ee2}
\|f \|^2_{L^2(-\infty,\mu)} - \| J f \|^2_{L^2(0,\mu)}
\le 2\lambda^{-\frac{1}{4}} \big(B^{\lambda,\mu}[f,f]+\|f \|^2_{L^2(-\infty,\mu)}\big).
\end{equation}

Now let us estimate the difference $C_N^\mu[Jf,Jf] - B^{\lambda,\mu}[f,f]$. For any $\varepsilon \in (0,1)$ and $a,b\in\RR$
one has $(a+b)^2\le (1+\varepsilon)a^2+2\varepsilon^{-1}b^2$. Therefore, for any $\delta>0$ and $\varepsilon \in (0,1)$
we have, with some $K>0$,
\begin{align*}
C_N^\mu[Jf,Jf] &= \int_0^\mu \big( f'(x) + f(0) e^{-x}\big)^2 \dd x + \int_0^\mu x^p \big( f(x) - f(0) e^{-x}\big)^2\dd x \\
&\leq (1+\varepsilon) \int_0^\mu \big(f'(x)^2+x^p f(x)^2\big)\dd x\\
&\qquad + 2 \varepsilon^{-1} f(0)^2 \int_0^\mu (1+x^p)e^{-2x}\dd x \\
&\leq (1+\varepsilon) \int_0^\mu \big(f'(x)^2+x^p f(x)^2\big)\dd x + K\varepsilon^{-1} f(0)^2 \\
&\leq (1+\varepsilon) \int_0^\mu \big(f'(x)^2+x^p f(x)^2\big)\dd x \\
& \qquad + K\varepsilon^{-1} \big( \delta\|f'\|^2_{L^2(-\infty,0)}+\delta^{-1} \|f\|^2_{L^2(-\infty,0)}\big).
\end{align*}
As previously, set $\delta:=\lambda^{-\frac{1}{2}}$ and $\varepsilon:=\lambda^{-\frac{1}{4}}$, then, with some $K'>0$,
\begin{align*}
C_N^\mu[Jf,Jf]&\le (1+\lambda^{-\frac{1}{4}}) \int_0^\mu \big(f'(x)^2+x^p f(x)^2\big)\dd x\\
&\qquad + K\lambda^{-\frac{1}{4}}\|f'\|^2_{L^2(-\infty,0)} +K\lambda^{-\frac{1}{4}}\cdot \lambda \|f\|^2_{L^2(-\infty,0)}\\
&\le
\int_0^\mu \big(f'(x)^2+x^p f(x)^2\big)\dd x\\
&\qquad 
+K'\lambda^{-\frac{1}{4}}\Big( \int_0^\mu \big(f'(x)^2+x^p f(x)^2\big)\dd x+\|f'\|^2_{L^2(-\infty,0)}\\
&\qquad\qquad+\lambda \|f\|^2_{L^2(-\infty,0)}\Big)\\
&\le B^{\lambda,\mu}[f,f]+K'\lambda^{-\frac{1}{4}}\Big(B^{\lambda,\mu}[f,f]+\|f \|^2_{L^2(-\infty,\mu)}\Big),
\end{align*}
resulting in
\begin{equation}\label{ee3}
C_N^\mu[Jf,Jf]-B^{\lambda,\mu}[f,f]\le K'\lambda^{-\frac{1}{4}}\big(B^{\lambda,\mu}[f,f]+\|f \|^2_{L^2(-\infty,\mu)}\big).
\end{equation}
By \eqref{ee2} and \eqref{ee3} we are in the situation of Proposition~\ref{Comp} with
\[
T:=B^{\lambda,\mu}, \quad T':=C_N^\mu, \quad \delta_1=2\lambda^{-\frac{1}{4}}, \quad \delta_2=K'\lambda^{-\frac{1}{4}}.
\]
Furthermore, in view of \eqref{ee1} one has $\Lambda_n(B^{\lambda,\mu})\le M:=\Lambda_n(C_D^1)$
for all $(\lambda,\mu)\in(0,+\infty)\times(1,+\infty)$. 
Therefore, one can find $\lambda_n>0$ such that
\[
\delta_1\big(1+\Lambda_n(T)\big)\equiv 2\lambda^{-\frac{1}{4}}\big(1+\Lambda_n(B^{\lambda,\mu})\big)\le 2(M+1)\lambda^{-\frac{1}{4}}
\]
for all $(\lambda,\mu)\in(\lambda_n,+\infty)\times(1,+\infty)$. Hence,  Proposition~\ref{Comp} implies
\begin{align*}
\Lambda_n(B^{\lambda,\mu}) &
\ge \Lambda_n(C_N^\mu)
- 
\frac{\big(2\Lambda_n(B^{\lambda,\mu})+K'\big) \lambda^{-\frac{1}{4}} \big(\Lambda_n(B^{\lambda,\mu})+ 1\big)}{1-2\lambda^{-\frac{1}{4}} \big(\Lambda_n(B^{\lambda,\mu})+ 1\big)}\\
& \ge \Lambda_n(C_N^\mu)
- 
\frac{(2M+K') (M+ 1)}{1-2\lambda_n^{-\frac{1}{4}}(M+ 1)}\, \lambda^{-\frac{1}{4}}\\
&=:\Lambda_n(C_N^\mu)-M_n \lambda^{-\frac{1}{4}}
\end{align*}
for all $(\lambda,\mu)\in(\lambda_n,+\infty)\times(1,+\infty)$.
\end{proof}

By combining Lemma~\ref{lem8} with the estimate of the eigenvalues of $C^\mu_N$ obtained in Lemma~\ref{lem7})
one arrives at the following result:
\begin{lemma}\label{lem9}
For any $n\in\NN$ there exist $m>0$ and $M>0$ such that
\[
\big| \Lambda_n(B^{\lambda,\mu})-\Lambda_n(A)\big|\le M (\lambda^{-\frac{1}{4}}+\mu^{-2})
\]
for all $(\lambda,\mu)\in(m,+\infty)\times(m,+\infty)$.
\end{lemma}

Now we can complete the proof of  Lemma~\ref{lemzh}. Choosing $\lambda=2^{\frac{2}{2+p}}\,h^{-\frac{2p}{2+p}}$ and $\mu=2^{\frac{1}{2+p}}h^{k-\frac{2}{2+p}}$
and using Lemma~\ref{lem6}, for $h\to 0^+$ we obtain
\begin{equation}
    \label{prox0}
\Lambda_n(Z_h) = 2^{\frac{2}{2+p}}h^{\frac{2p}{2+p}} \Lambda_n ( B^{\lambda,\mu}).
\end{equation}
By Lemma~\ref{lem9} we have
\[
\Lambda_n (B^{\lambda,\mu})=\Lambda_n(A)+\cO(\lambda^{-\frac{1}{4}}+\mu^{-2})\equiv \Lambda_n(A)+\cO(h^\frac{p}{4+2p}+h^{\frac{4}{2+p}-2 k}),
\]
and the substitution into \eqref{prox0} completes the proof of  Lemma~\ref{lemzh}.

\subsection{Proof of the lower eigenvalue bound}

We now use all the preceding components to obtain the sought lower bound for the eigenvalues of $R_h$ and then for those of $H_\alpha$.
For any $m>0$ we have $h_0^m=h^m(1-h^s)^{\frac{m}{2}}=h^m+\cO(h^{m+s})$, and then
we conclude by Lemma~\ref{lemzh} that
\begin{align*}
E_n(Z_{h_0})&=2^{\frac{2}{2+p}}E_n(A)\,h_0^{\frac{2p}{p+2}}+\cO(h_0^\frac{5p}{2p+4} + h_0^{2-2k})\\
&=2^{\frac{2}{2+p}}E_n(A)\,h^{\frac{2p}{p+2}}+\cO(h^{\frac{2p}{p+2}+s}+h^\frac{5p}{2p+4} + h^{2-2k}).
\end{align*}
The substitution into Lemma~\ref{lem15} gives then
\begin{gather*}
\Lambda_n(R_h)\ge -1+2^{\frac{2}{2+p}}E_n(A)\,h^{\frac{2p}{p+2}}+\rho(h),\\
\rho(h)=\cO(h^{\frac{2p}{p+2}+s}+h^\frac{5p}{2p+4} + h^{2-2k}+h^{2+2k(p-1)-s}+h^{2kp}).
\end{gather*}
It is convenient to set first $k=\frac{1}{p+1}$ to have
\[
\rho(h)=\cO(h^{\frac{2p}{p+2}+s}+h^\frac{5p}{2p+4} +h^{2+2\frac{p-1}{p+1}-s}+h^{\frac{2p}{1+p}}).
\]
Furthermore, choosing $s=1+\frac{p-1}{p+1}-\frac{p}{p+2}\equiv\frac{p(p+3)}{(p+1)(p+2)}$ we have
\begin{gather*}
\tfrac{2p}{p+2}+s=2+2\tfrac{p-1}{p+1}-s=\tfrac{p}{p+2}+1+\tfrac{p-1}{p+1}=\tfrac{p(3p+5)}{(p+1)(p+2)},\\
\rho(h)=\cO(h^{\frac{p(3p+5)}{(p+1)(p+2)}}+h^\frac{5p}{2p+4} +h^{\frac{2p}{1+p}}).
\end{gather*}
(One can prove that this choice of $s$ and $k$ optimizes the order in $h$.)
We compute then 
\[
\tfrac{p(3p+5)}{(p+1)(p+2)}- \tfrac{2p}{1+p}=\tfrac{p(3p+5)-2p(p+2)}{(p+1)(p+2)}=\tfrac{p^2+p}{(p+1)(p+2)}>0,
\]
which yields $h^{\frac{p(3p+5)}{(p+1)(p+2)}}=o(h^{\frac{2p}{1+p}})$ and $\rho(h)=\cO(h^\frac{5p}{2p+4} +h^{\frac{2p}{1+p}})$.
To summarize,
\[
\Lambda_n(R_h)\ge -1+2^{\frac{2}{2+p}}E_n(A)\,h^{\frac{2p}{p+2}}+\cO(h^\frac{5p}{2p+4} +h^{\frac{2p}{1+p}}).
\]
By Lemma~\ref{lem111} we have then, with a suitably small $\varepsilon>0$,
\[
\Lambda_n(F_{h,\varepsilon h^{\frac{1}{1-p}}})\ge -1+2^{\frac{2}{2+p}}E_n(A)\,h^{\frac{2p}{p+2}}+\cO(h^\frac{5p}{2p+4} +h^{\frac{2p}{1+p}}).
\]
Applying now Lemma~\ref{lemeps}, for $h:=\alpha^\frac{1-p}{p}$ and $\alpha\to+\infty$ we obtain
\begin{align*}
\Lambda_n(H_\alpha)&\ge \alpha^2\Lambda_n(F_{h,\varepsilon h^{\frac{1}{1-p}}})+\cO(1),\\
&\ge \alpha^2\Big(-1+2^{\frac{2}{2+p}}E_n(A)\,\alpha^{\frac{2(1-p)}{p+2}}+\cO(\alpha^\frac{5(1-p)}{2p+4} +\alpha^{\frac{2(1-p)}{1+p}})\Big)
+\cO(1)\\
&=-\alpha^2+2^{\frac{2}{2+p}}E_n(A)\,\alpha^{\frac{6}{p+2}}+\cO(\alpha^{\frac{13-p}{2p+4}}+\alpha^{\frac{4}{1+p}}).
\end{align*}
Noting that
\[
\eta_1:=\tfrac{6}{p+2}-\tfrac{13-p}{2p+4}=\tfrac{p-1}{2(p+2)}>0, \quad
\eta_2:=\tfrac{6}{p+2}-\tfrac{4}{p+1}=\tfrac{2(p-1)}{(p+1)(p+2)}>0
\]
we obtain
\[
\Lambda_n(H_\alpha)\ge -\alpha^2+2^{\frac{2}{2+p}}E_n(A)\,\alpha^{\frac{6}{p+2}}+\cO(\alpha^{\frac{6}{p+2}-\eta}),\quad \eta:=\min\{\eta_1,\eta_2\}>0.
\]
Recall that in Subsection~\ref{sec-upper} we already obtained a suitable upper bound and noted that $\Lambda_n(H_\alpha)$
is the $n$th eigenvalue of $H_\alpha$ if $\alpha$ is large. This completes the proof of Theorem~\ref{thm-main1}.

\renewcommand{\refname}

\end{document}